\documentclass[a4paper,11pt,oneside]{amsart}
\usepackage[top=2.2cm,left=2.8cm,right=2.8cm,bottom=2.8cm]{geometry}

\usepackage[foot]{amsaddr}

\usepackage{hyperref}
\hypersetup{
    colorlinks,
    citecolor=green,
    filecolor=black,
    linkcolor=blue,
    urlcolor=blue
}
\hypersetup{linktocpage}

\usepackage{calrsfs}

\usepackage{cite}
\usepackage{graphicx}

\usepackage{amsbsy}
\usepackage{latexsym}
\usepackage{amsfonts}
\usepackage{amssymb}
\usepackage[usenames]{color}
\usepackage{amsmath,amsthm}
\usepackage{enumerate}

\usepackage{tikz}
\usetikzlibrary{arrows}
\usepackage{mathdots}
\usepackage{mathtools}

\usepackage{stmaryrd}
\usepackage{dsfont}

\newcommand{\ii}{\mathrm{i}}

\newcommand{\plushalf}{ + \textstyle\frac12\displaystyle}

\newcommand\supp{\mathop{\rm supp}}

\providecommand{\abs}[1]{\lvert#1\rvert}
\providecommand{\bigabs}[1]{\bigl\lvert#1\bigr\rvert}
\providecommand{\Bigabs}[1]{\Bigl\lvert#1\Bigr\rvert}
\providecommand{\biggabs}[1]{\biggl\lvert#1\biggr\rvert}

\providecommand{\norm}[1]{\lVert#1\rVert}
\providecommand{\bignorm}[1]{\bigl\lVert#1\bigr\rVert}
\providecommand{\Bignorm}[1]{\Bigl\lVert#1\Bigr\rVert}

\newtheorem{theorem}{Theorem}
\newtheorem{lemma}[theorem]{Lemma}
\newtheorem{prop}[theorem]{Proposition}
\newtheorem{cor}[theorem]{Corollary}

\theoremstyle{definition}

\newtheorem{assumption}{Assumption}

\theoremstyle{remark}
\newtheorem{remark}[theorem]{Remark}

\newcommand{\cN}{{\mathcal{N}}}
\newcommand{\cF}{{\mathcal{F}}}

\newcommand{\sdd}{\,\mathrm{d}}

\newcommand{\EE}{\mathbb{E}}

\newcommand{\bbS}{\mathbb{S}}

\newcommand{\Chi}{\raise .3ex
\hbox{\large $\chi$}}

\newcommand{\R}{\mathbb{R}}
\newcommand{\N}{\mathbb{N}}
\newcommand{\Z}{\mathbb{Z}}

\newcommand{\dx}{\,\mathrm{d}x}
\newcommand{\dt}{\,\mathrm{d}t}
\newcommand{\dS}{\,\mathrm{d}\sigma}

\renewcommand{\Re}{\ensuremath{\mathrm{Re\,}}}
\renewcommand{\Im}{\ensuremath{\mathrm{Im\,}}}

\newcommand{\cY}{\mathsf{Y}}

\numberwithin{equation}{section}

\date{\today}

\title[Multilevel Representations of Random Fields on the Sphere]{Multilevel Representations of Isotropic Gaussian Random Fields on the Sphere}
\author{Markus Bachmayr$^1$}
\address{\rm $^1$ Institut f\"ur Mathematik, Johannes Gutenberg-Universit\"at Mainz, Staudingerweg 9, 55128 Mainz, Germany}
\email[Markus Bachmayr]{bachmayr@uni-mainz.de}
\author{Ana Djurdjevac$^2$}
\address{\rm $^2$ Institut f\"ur Mathematik, Freie Universit\"at Berlin, Arnimallee 6, 14195 Berlin, Germany}
\email[Ana Djurdjevac]{adjurdjevac@zedat.fu-berlin.de}
\thanks{
 The authors would like to thank the Isaac Newton Institute for
  Mathematical Sciences, Cambridge, for support and hospitality during
  the programme ``Uncertainty Quantification for Complex Systems'' where work on this paper was
  undertaken. This work was supported by EPSRC grant no EP/K032208/1. 
M.B.\ acknowledges funding by
the Deutsche Forschungsgemeinschaft (DFG) TRR 146 (project number 233630050).
A.Dj.\ was supported by the Deutsche Forschungsgemeinschaft (DFG, German Research Foundation) under Germany's Excellence Strategy -- The Berlin Mathematics
Research Center MATH+ and CRC 1114 ``Scaling Cascades in Complex Systems''.} 

\begin{document}

\begin{abstract}
	Series expansions of isotropic Gaussian random fields on $\mathbb{S}^2$ with independent Gaussian coefficients and localised basis functions are constructed. Such {representations with multilevel localised structure} provide an alternative to the standard Karhunen-Lo\`eve expansions of isotropic random fields in terms of spherical harmonics. The basis functions are obtained by applying the square root of the covariance operator to spherical needlets. Localisation of the resulting covariance-dependent multilevel basis is shown under decay conditions on the angular power spectrum of the random field. In addition, numerical illustrations are given and an application to random elliptic PDEs on the sphere is analysed.

\smallskip

\end{abstract}

\maketitle

\section{Introduction}

Random fields on the sphere $\bbS^2$ are an important tool in many disciplines where measurement data are defined on a sphere, for instance geophysics, climatology, oceanography, and astrophysics; for an overview, we refer to \cite{MP,hilbe2012astrostatistical,jeong2017spherical,porcu2018modeling}. 
A random field on $\bbS^2$ is a mapping $u: \Omega \times \bbS^2 \to \R$ that is $\mathcal{F} \otimes \mathcal{B}(\bbS^2)$-measurable, where $(\Omega, \mathcal{F}, \mathbb{P})$ is a probability space.
Such a random field on $\bbS^2$ is called \emph{Gaussian} if for every $k \in \N$ and $s_1,\ldots, s_k \in \bbS^2$, the vector $(u(s_1), \dots, u(s_k))$ has multivariate Gaussian distribution.
A prominent application of Gaussian random fields on $\bbS^2$ in cosmology is in the analysis of the cosmic microwave background radiation  \cite{mcewen2007cosmological}. In these applications \emph{isotropic} random fields, whose probability law is invariant under spherical rotations, play a central role.

Beyond the efficient sampling of random fields on $\bbS^2$ (see, for instance, \cite{creasey2018fast,emery2019turning,HKS}), 
many tasks in the analysis and computational treatment of random fields require their expansion as series of functions with random scalar coefficients.
Note that by subtracting $\EE[u]$ from any given random field $u$, we can restrict ourselves to the case of \emph{centred} random fields with $\EE[u] = 0$.
The classical Karhunen-Lo\`eve (KL) expansion of a centred Gaussian random field $u$ on $\bbS^2$ yields a random series representation 
\begin{equation}\label{genKL}
  u(s) = \sum_{i = 1}^\infty z_i \varphi_i(s), \quad s \in \bbS^2,
\end{equation}
where $(z_i)_{i \in \N}$ is a sequence of independent scalar Gaussian random variables.
The functions $\varphi_i \in L_2(\bbS^2)$ are determined as eigenfunctions of the covariance operator of $u$. 
While this expansion yields the most rapid convergence in $L_2$-norm, the functions $\varphi_i$ typically exhibit global oscillations.

However, alternative expansions in terms of basis functions with spatial localisation can be advantageous.
A classical example of such a localised expansion is the L\'evy-Ciesielsky representation \cite{C61} of the Brownian bridge $b$ on $[0,1]$, 
\begin{equation}\label{levyciesielsky}
    b(x) = \sum_{j = 0}^\infty \sum_{k = 0}^{2^j-1} y_{jk} \, 2^{-\frac{j}{2}} h( 2^j x - k), \quad x \in [0,1].
\end{equation}
Here $y_{jk} \sim \cN(0,1)$ are independent and $h(x) =\max\{  1 - 2 \abs{  x - \frac12 }, 0\}$ is a piecewise affine linear hat function supported on $[0,1]$, and thus $\abs{\supp h( 2^j x - k)} = 2^{-j}$. The representation \eqref{levyciesielsky} needs to be compared to the KL expansion of $b$, which reads 
\begin{equation}\label{bbkl}
   b(x) = \sum_{i = 1}^\infty y_i \frac{ \sin (\pi i x) }{\pi i}, 
\end{equation}
where $y_i \sim \cN(0,1)$. Both expansions describe the same random field in terms of independent scalar random variables, but the L\'evy-Ciesielsky representation is especially well-suited for studying the regularity of realisations in H\"older or more general Besov spaces \cite{C61,Roy}.

Multilevel representations in terms of basis functions having similar localisation as in \eqref{levyciesielsky} were constructed in \cite{BCM}
for a more general stationary Gaussian random fields $u$ on domains $D \subset \R^{{m}}$, {$m \in \N$}.
The covariance function of such stationary random fields is of the form
\begin{equation}\label{eq:stationarycov}
   \EE[ u(x) \,u(x') ] = k(x - x'), \quad x,x' \in D,
\end{equation}
for a function $k$ on $\R^{{m}}$ with non-negative Fourier transform $\hat k$, defined as
\[
   \hat k(\xi) =  \mathcal{F}[k](\xi) = \int_{\R^{{m}}} k(x)\, e^{-\ii \xi x} \,dx, \quad \xi \in \R^{{m}}.
\]
In \cite{BCM}, multilevel representations are obtained for a class of $k$ comprising in particular the family of Mat\'ern covariances. These covariances are given in terms of their Fourier transforms by
\[
 \hat k(\xi) = c_{\nu,\lambda}\, \biggl( \frac{2\nu}{\lambda^2} + \abs{\xi}^2  \biggr)^{-(\nu + d/2)}, \quad c_{\nu,\lambda} := \frac{ 2^d \pi^{d/2} \Gamma(\nu + d/2) (2\nu)^\nu}{ \Gamma(\nu) \lambda^{2\nu}},
\]
for $\xi \in \R^{{m}}$,
with parameters $\nu,\lambda>0$. The basis functions constructed in \cite{BCM} have similar properties as the classical Meyer wavelets.

In computational methods for partial differential equations (PDEs) with random fields as coefficients,  representations of random fields in terms of hierarchical multilevel basis functions as in \eqref{levyciesielsky} can also {have advantages.} {For sparse polynomial approximations of random PDEs on domains $D\subset \R^m$ with lognormally distributed diffusion coefficients (that is, coefficients of the form $\exp(b)$, where $b$ is a Gaussian random field), improved convergence rates have been obtained in \cite{BCDM} based on such multilevel expansions of $b$. 
For slightly simplified model problems where the random field expansion enters in the diffusion coefficient in an affine manner, multilevel expansions have been shown to enable the construction of adaptive stochastic Galerkin methods of near-optimal computational complexity \cite{BCD,BV} for the computation of sparse polynomial approximations. In contrast, algorithms with these near-optimality properties are not available for KL-type expansions of random fields.}
The choice of a series expansion can thus be interpreted as a choice of coordinates for the random field that has computational implications.
{In a similar vein, in \cite{HS19a,HS19b} and independently in \cite{K19}, multilevel expansions have served as the basis of improved Quasi-Monte Carlo (QMC) integration methods for lognormal random PDEs.}

In this work, we consider expansions of {Gaussian} random fields on the sphere that have similar hierarchical multilevel structure as in previous constructions on domains, where we focus on the case of \emph{isotropic} random fields on $\bbS^2$.
The random field $u$ is called (strongly) isotropic if for every $k \in \N$, $s_1, \dots, s_k \in \bbS^2$ and $g \in \mathrm{SO}(3) $, the random vectors $(u(s_1), \dots, u(s_k))$ and $(u(gs_1), \dots, u(gs_k))$ have the same law.
The random field $u$ is said to be 2-weakly isotropic if $u(s) \in L_2(\Omega)$ for every $s \in \bbS^2$ and  
\[
\EE[u(s_1)] = \EE[u(gs_1)], \quad
\EE[u(s_1)\, u(s_2)] = \EE[ u(g s_1) \, u(g s_2)]
\]
for all $s_1, s_2 \in \bbS^2$ and all $g \in \mathrm{SO}(3)$.
Gaussian random fields are isotropic if and only if they are 2-weakly isotropic \cite[Prop.~5.10(3)]{MP}; in other words, they are isotropic as soon as their expectations and covariance functions are invariant under rotations. 
In the case of centred $u$ that we consider, this reduces to the condition that there exists a function $\rho\colon [-1,1] \to \R$ such that 
\begin{equation}\label{rhodef0}
  \rho(s_1 \cdot s_2) = \EE[ u(s_1)\,u(s_2)],
\end{equation}
where the inner product $s_1 \cdot s_2$ equals the cosine of the angle between $s_1$ and $s_2$.

The KL expansion of a centred isotropic Gaussian random field yields a decomposition in terms of spherical harmonics $Y_{\ell m}$, $\ell \geq 0$, $m = -\ell, \ldots, \ell$, which are eigenfunctions of the Laplace-Beltrami operator on $\bbS^2$. The expansion \eqref{genKL} thus takes the form
\begin{equation}\label{KLsph}
   u = \sum_{\ell = 0}^\infty \sum_{m=-\infty}^\infty z_{\ell m} Y_{\ell m},
\end{equation}
with $z_{\ell m} \sim \cN ( 0, A_\ell)$, where the positive real sequence $\mathsf{A} = (A_\ell)_{\ell \in \N_0}$ is called the \emph{power spectrum} of $u$. 
For a detailed study of this expansion, see \cite{LS}.

The alternative type of expansion with localised basis functions that we consider is based on spherical needlets \cite{NPW}. These are functions $\psi_{j k}$ with a scale parameter (or level) $j \in \N_0$ and an angular index $k \in \{1,\ldots, n_j\}$, which have localisation properties of the following type:
with the geodesic distance
\begin{equation}\label{geodesic}
  d(s, s') = \arccos(s \cdot s')
\end{equation}
on $\bbS^2$,
for each $\psi_{jk}$ there exists a point $\xi_{jk} \in \bbS^2$ such that 
\begin{equation}\label{localizationintro}
     \abs{\psi_{jk}(s)} \leq \frac{C 2^j }{ 1 + \bigl(2^j d(s, \xi_{jk})\bigr)^r  }, \quad s \in \bbS^2,
   \end{equation}
for some (or, depending on the precise construction, \emph{any}) $r \in \N$ with $C = C(r) > 0$; that is, $\psi_{jk}$ is concentrated near $\xi_{jk}$ and decays rapidly with increasing angular distance to this point.

As shown in \cite{NPW}, the points $\xi_{jk}$ can be chosen constructively such that the family $\{ \psi_{jk}\colon j \in \N_0, k = 1,\ldots,n_j \}$ is a \emph{Parseval frame} of $L_2(\bbS^2)$. This means that for any $f \in L_2(\bbS^2)$,
\[
   \norm{f}_{L_2(\bbS^2)}^2 = \sum_{j,k} \abs{\langle f, \psi_{jk}\rangle}^2
  \quad\text{and}\quad
  f = \sum_{j,k} \langle f, \psi_{jk}\rangle_{L_2(\bbS^2)} \psi_{jk},
\]
as for orthonormal bases of $L_2(\bbS^2)$, but without any requirement of linear independence of frame elements.
In \cite{BKMP}, it was shown that the needlet coefficients $\langle u, \psi_{jk}\rangle_{L_2(\bbS^2)}$ of a weakly isotropic (not necessarily Gaussian) second-order random field $u$ on $\bbS^2$, which are scalar random variables, are \emph{asymptotically uncorrelated} in the following sense: under certain decay conditions on the power spectrum $\mathsf{A} = (A_\ell)_{\ell \geq 0}$, for $k, k' \in \{ 1,\ldots,n_j\}$ one has
\begin{equation}\label{introcorrest}
\begin{aligned}
 \operatorname{Corr} \bigl( \langle u, \psi_{jk} \rangle_{L_2} , \langle u, \psi_{jk'} \rangle_{L_2} \bigr)
  &= \frac{ \EE [ \langle u, \psi_{jk} \rangle_{L_2} \langle u, \psi_{jk'} \rangle_{L_2} ]}{ \sqrt{ \EE \abs{\langle u, \psi_{jk} \rangle_{L_2}}^2 \,\EE  \abs{\langle u, \psi_{jk'} \rangle_{L_2}}^2 } } \\
  &\leq \frac{C_M}{ 1 + \bigl(2^j d( \xi_{jk} , \xi_{jk'})\bigr)^{M}  }
\end{aligned}
\end{equation}
for an $M\in \N$ and $C_M>0$. The correlation of coefficients corresponding to needlets with some fixed angular separation thus decreases rapidly with increasing level $j$.

We show that for any given isotropic Gaussian random field $u$, there exist modified needlets $\psi^\mathsf{A}_{jk}$ such that one has the expansion
\begin{equation}\label{introneedletexpansion}
  u = \sum_{j=0}^\infty \sum_{k = 1}^{n_j} y_{jk}  \psi^\mathsf{A}_{jk}
\end{equation} 
in terms of independent scalar Gaussian random variables $y_{jk} \sim \cN(0,1)$. This provides an alternative to the KL expansion \eqref{KLsph}, and the expansion coefficients are exactly uncorrelated: $\operatorname{Corr}( y_{jk} \, y_{j'k'} ) = 0$ for all $j,j'$ and $k,k'$ such that $j \neq j'$ or $k\neq k'$. The modified needlets $\psi^\mathsf{A}_{jk}$ still have the same localisation property \eqref{localizationintro} as the standard needlets $\psi_{jk}$, but with $r$ limited by certain features of the power spectrum of $u$. In other words, whereas coefficients in the expansion with respect to standard needlets $\psi_{jk}$ are \emph{dependent} (although \emph{asymptotically} uncorrelated) random variables, the modified needlets $\psi^\mathsf{A}_{jk}$ are adapted to the given random field to give, as in the KL case, a series expansion with \emph{independent} coefficients. This independence property is crucial in many applications in uncertainty quantification, for instance in the sparse polynomial approximation of random fields derived from $u$ as considered at the end of this work.

A necessary and sufficient condition on $\psi^\mathsf{A}_{jk}$ for the expansion \eqref{introneedletexpansion} in terms of independent scalar random variables $y_{jk}$ to hold was established in a more general setting in \cite{LP}. This condition is related to the \emph{reproducing kernel Hilbert space} (also known as Cameron-Martin space), denoted $\mathcal{H}$, of the random field $u$, which is defined as follows: the $\mathcal{H}$-inner product is first defined for finite linear combinations of the form 
\begin{equation}\label{kernelsum}
  \sum_{n = 0}^N \alpha_i \rho_{s_i}, \qquad s_1,\ldots, s_N \in \bbS^2, \quad \alpha_1, \ldots \alpha_N \in \R,\quad N \in \N,
\end{equation}
with $\rho_s(t) := \rho( s \cdot t)$ as
\[
   \langle \rho_{s}, \rho_{s'} \rangle_{\mathcal{H}} := \rho( s \cdot s'),
\]
and $\mathcal{H}$ is obtained as the closure of expressions \eqref{kernelsum} with respect to this inner product.
As shown in \cite{LP}, one has the expansion \eqref{introneedletexpansion} precisely when the family $\{ \psi^\mathsf{A}_{jk} \}$ is a Parseval frame of $\mathcal{H}$. 

Our strategy is thus to modify the needlets $\psi_{jk}$ such that they form a Parseval frame of $\mathcal{H}$. This can be achieved in explicit form by a transformation of their spherical harmonics representations using the power spectrum of the random field; the evaluation of the resulting functions is no more complicated than for standard spherical needlets. After collecting some preliminaries on spherical harmonics expansions in Section \ref{sec:sph}, the construction of the modified needlets $\psi^\mathsf{A}_{jk}$ is described in Section \ref{sec:needlets}. 

Whereas obtaining an expansion \eqref{introneedletexpansion} does not require any further assumptions on the power spectrum, the main issue lies in ensuring that decay properties of the form \eqref{localizationintro} still hold for $\psi^\mathsf{A}_{jk}$. 
As our main result, we show in Section \ref{sec:loc} that with some $C>0$, $\psi^\mathsf{A}_{jk}$ satisfies the angular decay estimate 
\begin{equation}\label{localizationmod}
    \abs{\psi^\mathsf{A}_{jk}(s)} \leq \frac{C 2^{- \beta j} }{ 1 + \bigl(2^j d(s, \xi_{jk})\bigr)^r  }, \quad s \in \bbS^2.
\end{equation}
Here $\beta>0$ and the exponent $r$ is restricted by the decay of the sequence $(\sqrt{A_\ell})_{\ell \in \N_0}$ and its forward differences, which are defined recursively by $\Delta_\ell^0 := \sqrt{A_\ell}$ and $\Delta_\ell^i := \Delta^{i-1}_{\ell+1} - \Delta^{i-1}_{\ell}$ for $i \in \N$. More precisely, we show that \eqref{localizationmod} holds provided that $\abs{\Delta^0_\ell} \to 0$ and that
\[
   \abs{\Delta^r_\ell} \leq C_r ( 1 + \ell)^{-(1 + \beta + r)}, \quad \ell \geq 0,
\]
for some $C_r>0$.
As shown below, this also implies that $ \abs{\Delta^i_\ell} \leq C_i ( 1 + \ell)^{-(1 + \beta + i)}$ for $i = 0,\ldots,r-1$ with $C_i>0$. Note that the parameter $\beta$ corresponds to the regularity of the realisations of the random field, whereas $r$ is related to the order up to which the decay of the differences of the $\sqrt{A_\ell}$ is consistent with the derivatives of the function $\ell \mapsto ( 1 + \ell )^{- ( 1 + \beta)}$.

In Section \ref{sec:num}, we provide numerical illustrations of the constructed expansions and on the dependence of their localisation properties on the power spectra of the random fields. {In addition, we study the approximation of the individual needlets by splines.}
In Section \ref{sec:pde}, we consider {two applications of the expansions: sparse polynomial approximations of elliptic PDEs with random coefficients on $\bbS^2$, and approximation and sampling of random fields by truncated expansions.}

{Let us note that our results can also be applied immediately in the context of the QMC methods in \cite{HS19a,HS19b,K19} that make use of localisation in random field expansions. Moreover, the expansions contructed here may be a suitable tool for the study of Besov regularity of realisations of random fields (complementing results on H\"older regularity in \cite{HLS18}), similarly to the role played by the expansion \eqref{levyciesielsky} of the Brownian bridge in \cite{Roy}.}

\begin{remark}
 Analogous results can be shown for random fields on $\bbS^n$ with $n>2$ by a similar adaptation of techniques in \cite{NPW}.  {The present work can also directly be extended to Gaussian random fields on product manifolds $[0,T]\times \bbS^2$ corresponding to an additional time dependence, provided that the covariance of the Gaussian field has product structure and is isotropic in the spherical variable.
 The extension to general random fields that are stationary in time and isotropic on the sphere, as considered in \cite{de2018regularity}, is left for future work.}
\end{remark}

\section{Karhunen-Lo\`eve Expansions of Random Fields on $\bbS^2$}\label{sec:sph}

In this section, we collect some basic facts on spherical harmonics and on their role in series expansions of random fields on $\bbS^2$; for further details, we refer to \cite{MP,LS}.
An important role in the definition of spherical harmonics is played by the \emph{associated Legendre functions} $P_{\ell m}$, which are defined as
\[
  P_{\ell m}(x) = (-1)^m (1-x^2)^{m/2} \frac{d^m}{dx^m} P_\ell(x),
\]
where $P_\ell$ is the $\ell$-th \emph{Legendre polynomial} with explicit expression
\[
  P_\ell(x) = \frac{1}{2^\ell \ell!} \frac{d^\ell}{dx^\ell} \bigl(x^2 - 1)^\ell\,.
\] 
The Legendre polynomials satisfy 
\begin{equation}\label{legendrenormalization}
\max_{x\in[-1,1]} \abs{P_\ell(x)} = 1,\quad \ell \in \N_0, 
\end{equation}
as well as
\[
   \int_{-1}^1 P_\ell(x)\, P_{\ell'}(x)\dx = \frac{2}{2\ell + 1}\delta_{\ell,\ell'}, \quad \ell, \ell' \in \N_0.
\]
In particular, note that $P_{\ell 0} = P_\ell$.

In what follows, for $s \in \bbS^2$, we use the convention
\begin{equation}\label{sphericalcoord}
  s = (\sin \vartheta \cos \varphi, \sin \vartheta \sin \varphi, \cos\vartheta) \quad\text{ with $\vartheta \in [0,\pi]$ and $\varphi \in [0,2\pi)$}
\end{equation}  
for writing $s$ in spherical coordinates.
The \emph{real-valued spherical harmonics} are now defined as
\[
\begin{aligned}
  Y_{\ell m}(s) &=  \sqrt{2} N_{\ell m}  P_{\ell m}(\cos \vartheta) \cos( m \varphi), \quad & m & = 1, \ldots, \ell, \\
  Y_{\ell 0}(s) &=  N_{\ell 0} P_{\ell}(\cos \vartheta),  & &\\
  Y_{\ell m}(s) &=  \sqrt{2} N_{\ell m} P_{\ell \abs{m}}(\cos \vartheta) \sin(\abs{m} \varphi),  & m & = -\ell, \ldots, -1,
\end{aligned}
\]
with the normalisation factors
\[
   N_{\ell m} := \sqrt{\frac{2\ell +1}{4\pi} \frac{(\ell- \abs{m})!}{(\ell + \abs{m})!} } \,.
\]
{Spherical harmonics can also be characterised as restrictions to the unit sphere $\bbS^2$ of real harmonic polynomials in $\mathbb{R}^3$ (see, e.g., \cite[Sec.~3.4, Proposition 3.33]{MP}).}

It is well known (see, e.g., \cite[Sec.~3.4]{MP}) that the family $\{ Y_{\ell m}\colon \ell \in \N\cup\{0\}, m \in \{-\ell, \ldots, \ell\} \}$ is an orthonormal basis of $L_2(\mathbb{S}^2)$. In particular,
\[
    f = \sum_{\ell =0}^\infty \sum_{m = -\ell}^\ell \left( \int_{\bbS^2} Y_{\ell m} \, f\dS \right) Y_{\ell m}\,, \qquad f \in L_2(\bbS^2),
\]
where we write $\sigma$ for the spherical measure with $\sigma(\bbS^2) = 4\pi$ corresponding to the surface area of the unit sphere.

In what follows, we frequently use the well-known fact that for any $\ell \in \N_0$,
\begin{equation}\label{Ysumlegendre}
 \sum_{m=-\ell}^\ell Y_{\ell m}(x) Y_{\ell m}(y) =
   \frac{2\ell + 1}{4\pi} P_\ell (x \cdot y), \qquad x, y \in \bbS^2.
\end{equation}

\begin{remark}
The \emph{complex-valued spherical harmonics} are 
\[
\begin{aligned}
\cY_{\ell m}(s) &= N_{\ell m} P_{\ell m}(\cos \vartheta) e^{\mathrm{i} m \varphi}, \quad &  m&\geq 0 , \\
\cY_{\ell m}(s) &= (-1)^m \overline{\cY}_{\ell,-m} (s),  & m &< 0.
\end{aligned}
\]
From this, one recovers the real-valued spherical harmonics by 
\[
\begin{aligned}
  Y_{\ell m} &:= \sqrt{2}\, \Re \cY_{\ell m} \text{ for $m = 1,\ldots, \ell$,} \\
  Y_{\ell, -m} &:= \sqrt{2}\, \Im \cY_{\ell m} \text{ for $m = 1,\ldots, \ell$,}
\end{aligned}
\] 
see also \cite[Rem.~3.24, Rem.~3.37]{MP}. 
Using $\cY_{\ell m} = (-1)^m \overline{\cY_{\ell,-m}}$, it is easy to see that
\begin{equation}\label{Ysumidentity}
\sum_{m=-\ell}^\ell \cY_{\ell m}(\xi_{jk}) \overline{ \cY_{\ell m} (s) } 
   =  \sum_{m=-\ell }^\ell Y_{\ell m}(\xi_{jk}) Y_{\ell m}(s) .
 \end{equation}
\end{remark}

The Karhunen-Lo\`eve expansion of a real-valued centred isotropic Gaussian random field $u$ on $\bbS^2$ is of the form
\begin{equation}\label{isotropicreal}
    u = \sum_{\ell = 0}^\infty \sqrt{A_\ell}  \sum_{m = -\ell}^\ell y_{\ell m}\, Y_{\ell m},\qquad y_{\ell,m}\sim \cN(0,1)\,\text{ i.i.d.,}
\end{equation}
with the real-valued spherical harmonics $Y_{\ell m}$.
Here the positive sequence $\mathsf{A} = (A_\ell)_{\ell \in \N_0}$, the power spectrum of $u$, is such that
 \begin{equation}\label{Asummability}
   \sum_{\ell=0}^\infty (2\ell + 1) A_\ell < \infty.
 \end{equation}
This representation was used, for instance, in \cite{KLG} for the discretization of stochastic differential equations on the sphere. The covariance function of $u$ is given, for $s, s' \in \bbS^2$, by
\begin{equation}\label{rhodef}
\begin{aligned}
 \rho(s \cdot s')  & := \EE[ u(s) \, u(s')] \\
  & = \sum_{\ell = 0}^\infty A_\ell \sum_{m = -\ell}^\ell Y_{\ell m}(s)\, Y_{\ell m}(s') = \sum_{\ell = 0}^\infty A_\ell  \frac{2\ell + 1}{4\pi} P_\ell (s \cdot s') ,
 \end{aligned}
\end{equation}
where we have used \eqref{Ysumlegendre}.
The condition \eqref{Asummability} implies that the integral operator on $L_2(\bbS^2)$ with kernel $\rho$,
\begin{equation}\label{covop}
\begin{aligned}
  T\colon L_2(\bbS^2) &\to L_2(\bbS^2),\\ f &\mapsto \left( s \mapsto \int_{\bbS^2} \rho(s\cdot s')\,f(s')\dS(s') 
 \right)
   \end{aligned}
\end{equation}
where 
\[
\int_{\bbS^2} \rho(s\cdot s')\,f(s')\dS(s') 
   =  \sum_{\ell=0}^\infty A_\ell \sum_{m=-\ell}^\ell  \left( \int_{\bbS^2} Y_{\ell m}(s') \,f(s')\dS(s')  \right) Y_{\ell m}(s) 
\]
 is a nuclear operator.
We also have $T\colon C(\bbS^2)' \to C(\bbS^2)$, since by \eqref{legendrenormalization} and \eqref{Asummability}, the series of continuous functions of $s,s'$ on the right side of \eqref{rhodef} converges uniformly, and thus the kernel $\rho$ is continuous.

\begin{remark}
In terms of complex-valued spherical harmonics, such random fields can alternatively be written in the form 
\begin{equation}\label{isotropicrf}
   u = \sum_{\ell = 0}^\infty \sum_{m=-\ell}^\ell a_{\ell m} \cY_{\ell m};
\end{equation}
here $a_{\ell m}$ are complex-valued Gaussian random variables where $a_{\ell m} \sim \cN(0, A_\ell /2)$ are independent for $m > 0$, $a_{\ell 0} \sim \cN(0,A_\ell)$ is real-valued and independent of the former, and $\Re a_{\ell m} = (-1)^m \Re a_{\ell ,-m}$, $\Im a_{\ell m} = (-1)^{m+1} \Im a_{\ell, -m }$; see \cite{MP,LS}.
\end{remark}

\begin{remark}\label{rem:regularity}
As shown in \cite[Thm.~4.6]{LS}, H\"older regularity of realisations holds under a slightly stronger condition than \eqref{Asummability}: if for a $\beta>0$,
\begin{equation}
\label{regularitysumcond}
   \sum_{\ell = 0}^\infty (1 + \ell)^{1 + 2\beta} A_\ell  < \infty,
\end{equation}
 then the random field has a modification that is in $C^{\gamma}(\bbS^2)$ for any $\gamma < \beta$.
\end{remark}

\section{Needlet expansions of random fields}\label{sec:needlets}

For many purposes, the notion of orthonormal bases is too restrictive. 
However, one can instead consider frames, where the basis functions are no longer required to be linearly independent. A family $\{ \varphi_i \}_{i \in \N}$ is called a frame for the Hilbert space $\mathcal{H}$ if
\[
    c \norm{ f }_{\mathcal{H}}^2 \leq \sum_{i=1}^\infty \abs{ \langle f, \varphi_i \rangle_{\mathcal{H}}}^2 \leq C \norm{f}_{\mathcal{H}}^2, \quad
     f \in \mathcal{H},
\]
for some $c,C>0$. If $c = C = 1$, the frame is called a Parseval frame, and one has 
\[
   f = \sum_{i=1}^\infty \langle f, \varphi_i\rangle_{\mathcal{H}} \, \varphi_i, \quad f \in \mathcal{H},
\]
with convergence in $\mathcal{H}$. 

The {needlet functions} that have been constructed in \cite{NPW} are a Parseval frame of $L_2(\bbS^2)$ and additionally satisfy localisation properties, that is, each needlet takes very small values sufficiently far away from a certain point on $\bbS^2$.
This is in contrast to the spherical harmonics, where each basis function has global oscillations on $\bbS^2$. One major advantage of basis functions $\varphi_i$ with localisation is that the coefficients $\langle f, \varphi_i\rangle_{\mathcal{H}}$ depend essentially only on local features of $f$. This means that one can use such frames for efficient adaptive approximation of functions that are smooth up to localised singularities.

The construction of spherical needlet frames relies crucially on suitable quadratures on the sphere. The existence of quadratures satisfying the following set of assumptions was shown in \cite{NPW}.

\begin{assumption}\label{assLambda}
For any $j \in \N_0$, let $\Lambda_j = \{ (\lambda_{j k}, \xi_{jk})\colon k = 1,\ldots, n_j\}$ with $n_j \in \N$, where $\lambda_{jk}>0$ and $\xi_{jk} \in \bbS^2$ are quadrature weights and points, respectively, such that
\begin{equation}\label{quadexactness}
\int_{\bbS^2} f \dS = \sum_{k=1}^{n_j} \lambda_{jk} f(\xi_{jk})
\quad\begin{array}{l} \text{for any polynomial $f$} \\ \text{of degree at most $2(2^{j} -1)$.} \end{array}
\end{equation}
In addition, let $n_j \leq C 2^{2j}$ with a $C>0$ independent of $j$, let
\begin{equation}\label{hdecrease}
   \frac{h_j}{4} \leq  h_{j+1} \leq \frac{h_j}{2} , \quad h_j = \sup_{s \in \bbS^2} \min \bigl\{ d(s, \xi_{jk})\colon k = 1,\ldots,n_j \bigr\},
\end{equation}
where $d$ denotes the geodesic distance on $\bbS^2$ as in \eqref{geodesic}, as well as
\begin{equation}\label{meshratio}
   h_j \leq \min_{k \neq k'} d(  \xi_{jk}, \xi_{jk'} ); 
\end{equation}
and let the corresponding quadrature weights satisfy
\begin{equation}\label{lambdabound}
  \lambda_{j k} \leq c 2^{-2j}
 \end{equation}
 with a $c > 0$ independent of $j$.
\end{assumption}

 For the set of indices of quadrature points and weights associated to such a quadrature, we write
 \[
   \mathcal{J} := \bigl\{  (j,k) \colon j \in \N_0, \, k = 1,\ldots, n_j  \bigr\}.
 \]
That the quadrature points can be chosen to satisfy \eqref{hdecrease} and \eqref{meshratio} {follows from \cite[Prop.~2.1]{NSWW07} (as restated in \cite[Prop.~2.1]{NPW}).}
The existence of corresponding positive weights satisfying \eqref{quadexactness} and \eqref{lambdabound} is a consequence of \cite[Cor.~4.4]{NPW}, {which improves on earlier results in \cite{MNW01}.}

\begin{remark}{
  Note that the quadrature points and weights are independent of the considered random field and can thus be precomputed in a numerical implementation. The proof of \cite[Prop.~2.1]{NSWW07} contains a procedure for finding suitable point sets that can in principle be realized numerically. However, each iteration of this algorithm requires finding a point on $\bbS^2$ having a fixed distance to the previously determined points, which exists by compactness, and therefore this scheme is not fully constructive. Once the quadrature points as in Assumption \ref{assLambda} are determined, corresponding weights can be computed by solving a quadratic programming problem as described in \cite[Sec.~4.3]{MNW01}. An alternative construction are \emph{spherical $t$-designs}, that is, sets of points on $\bbS^2$ for which the polynomial exactness property \eqref{quadexactness} up to degree $t$ holds with equal weights $\lambda_{jk} = 4\pi / n_j$. Such point sets have been computed numerically up to large $t$ in \cite{W} and have been used in the construction of standard spherical needlets in \cite{le2020isotropic}.}
\end{remark}

A second ingredient in the definition of spherical needlets is a suitable cutoff function $\kappa$ with the following properties.

\begin{assumption}\label{asskappa}
Let $\kappa \in C^r(\R)$ with $r > 2$ such that $\supp \kappa = [\frac12, 2]$ and 
\begin{equation}\label{eq:partunitycond}
  \abs{ \kappa(t) }^2  +  \abs{ \kappa(2t) }^2 = 1, \quad t \in [\textstyle\frac12\displaystyle, 1].
\end{equation}
\end{assumption}

{An example of such $\kappa$ is given in \eqref{eq:kappaexample}.}
With $\kappa$ as in Assumption \ref{asskappa}, let $b_j(t) = \kappa(2^{-j} (2t + 1))$ for $j\in \N_0$.
Following \cite{NPW}, the spherical needlets corresponding to $\kappa$ and $(\Lambda_j)_{j\in \N_0}$ are defined as
\begin{equation}\label{needletdef1}
 \psi_{j k} (s)
  = \sqrt{\lambda_{jk}} \sum_{\ell = 0}^\infty b_j(\ell)  \sum_{m=-\ell}^\ell Y_{\ell m}(\xi_{jk}) Y_{\ell m}(s) ,\quad (j,k) \in \mathcal{J}.
\end{equation}
By \eqref{Ysumlegendre}, we also have the simplified form
\begin{equation}\label{needletdef3}
   \psi_{jk} (s) = \sqrt{\lambda_{jk}} \sum_{\ell = 0}^\infty b_j(\ell)   \frac{2\ell + 1}{4\pi} P_\ell (s \cdot \xi_{jk}).
\end{equation}

\begin{remark}
 By \eqref{Ysumidentity}, the needlets can be rewritten in terms of complex spherical harmonics as
 \begin{equation}\label{needletdef2}
 \psi_{j k} (s) = \sqrt{\lambda_{jk}} \sum_{\ell = 0}^\infty b_j(\ell) \sum_{m=-\ell}^\ell \cY_{\ell m}(\xi_{jk}) \overline{ \cY_{\ell m} (s) }  \,,
\end{equation}
as in the original definition in \cite{NPW}.
\end{remark}

In addition to forming a Parseval frame, needlets also have vanishing moments; that is, as can be directly seen from the construction, they are orthogonal to spherical harmonics up to certain values of $\ell$, and thus to polynomials of the corresponding degrees. This is summarised in the following adaptation of \cite[Thm.~5.2]{NPW}. That the degree of exactness of the underlying quadrature stated in \eqref{quadexactness} suffices in our particular case can be seen from the discussion preceding this theorem in \cite[Sec.~5]{NPW}.

\begin{theorem}[\!\!{\cite[Thm.~5.2]{NPW}}]
\label{thm:needletframe}
Let Assumptions \ref{assLambda} and \ref{asskappa} hold.
Then the family $\{ \psi_{jk}  \}_{(j,k) \in \mathcal{J}}$ with $\psi_{jk}$ as in \eqref{needletdef1} is a Parseval frame of $L_2(\bbS^2)$. Moreover, $\psi_{jk}$ are orthogonal in $L_2(\bbS^2)$ between non-adjacent levels $j$ and to polynomials of degree less than $\frac12 (2^{j-1} -1)$  when $j  > 1$.
\end{theorem}

What distinguishes needlet frames from other types of orthogonal bases or frames on the sphere are the strong localisation properties of the individual needlets. This localisation can be quantified in terms of the angle with respect to the point $\xi_{jk}$ associated to a given needlet $\psi_{jk}$ as follows.

\begin{theorem}[\!\!{{\cite[Cor.~5.3]{NPW}}}]
 Let Assumptions \ref{assLambda} and \ref{asskappa} hold. 
 Then there exists $C>0$ such that for $(j, k) \in \mathcal{J}$,
 \begin{equation}\label{localizationestimate}
  \abs{\psi_{jk}(s)} \leq  \frac{C 2^{j}}{1 + (2^j \theta)^r}
   \quad\text{for all $s \in \bbS^2$, with $\theta:=d (s, \xi_{jk}) \in [0,\pi]$,}
 \end{equation}
 {with $r$ as in Assumption \ref{asskappa}.}
\end{theorem}

For a given centred Gaussian random field $u$, we now transform the standard needlets $\psi_{jk}$ to modified needlets $ \psi^\mathsf{A}_{jk}$ such that 
\begin{equation}\label{modneedletexpansion}
  u = \sum_{(j,k) \in \mathcal{J}} y_{j,k}  \psi^\mathsf{A}_{jk},\quad
  \text{$y_{j,k} \sim \cN(0,1)$ i.i.d.}
\end{equation}
As the following result shows, such a series representation with uncorrelated standard Gaussian coefficients holds precisely when $( \psi^\mathsf{A}_{jk})_{(j,k)\in\mathcal{J}}$ is a Parseval frame of the reproducing kernel Hilbert space associated to the random field.

\begin{theorem}[{\!\!\cite[{Thm.~1}]{LP}}]
\label{thm:frame vs expansion}
Let $u$ be a centred Gaussian random field with reproducing kernel Hilbert space $\mathcal{H}$. The following are equivalent for $\Phi:=\{ \varphi_i \}_{i \in \N}$ with $\varphi_i \in \mathcal{H}$:
\begin{enumerate}[{\upshape (i)}]
\item The family $\Phi$ is a Parseval frame of $\mathcal{H}$, that is, 
\[
   \norm{ f}_{\mathcal{H}}^2 = \sum_{ i =1 }^\infty \abs{\langle f, \varphi_i \rangle}^2\quad\text{for all $f \in \mathcal{H}$.}
\]
\item One has the representation
\[
   u = \sum_{i = 1}^\infty y_i \varphi_i
\]
where $y_i \sim \cN(0,1)$ are independent identically distributed.
\end{enumerate}
\end{theorem}

Such a Parseval frame of the reproducing kernel Hilbert space associated to the random field can be obtained by applying a suitable factorization of its covariance operator to a frame in a reference Hilbert space, which in our case will be $L_2(\bbS^2)$.

\begin{prop}[{\!\!\cite[Prop.~1]{LP}}]
\label{prop:factorization}
Let $u$ be a Gaussian random field with realisations in the separable Banach space $\mathcal{E}$ with covariance operator $T\colon \mathcal{E}'\to \mathcal{E}$. If $T = S S'$ with $S \colon \mathcal{K}\to \mathcal{E}$, where $\mathcal{K}$ is a separable Hilbert space, and if $\{ \varphi_i \}_{i \in \N}$ is a Parseval frame of $\mathcal{K}$, then $\{ S \varphi_i \}_{i \in \N}$ is a Parseval frame of the reproducing kernel Hilbert space of $u$.
\end{prop}

This use of factorizations can be regarded as the generalisation of the observation that if a symmetric positive semidefinite matrix $M \in \R^{n\times n}$ has the factorization $M = L L^T$, then for $z \sim \mathcal{N}(0, I)$ one has $L z \sim \mathcal{N}(0,M)$.
In the case of $T$ as in \eqref{covop}, a factorization as in Proposition \ref{prop:factorization} can immediately be obtained from \eqref{covop}: since $T$ is diagonal in the basis of spherical harmonics, we have $T = S^2$ with
\begin{equation}\label{Sdef}
   S f := \sum_{\ell = 0}^\infty \sqrt{A_\ell} \sum_{m = -\ell}^\ell \left(  \int_{\bbS^2} Y_{\ell m} f \dS  \right)  Y_{\ell m}\,.
\end{equation}
Here we can apply Proposition \ref{prop:factorization} to $T\colon C(\bbS^2)' \to C(\bbS^2)$ and $S\colon L_2(\bbS^2)\to C(\bbS^2)$.

We define the following modified needlets, which are a Parseval frame of the reproducing kernel Hilbert space of the random field $u$ in \eqref{isotropicrf}:
For a given power spectrum $\mathsf{A}$, we now define modified needlets by
\begin{equation}\label{modifiedneedlets}
 \psi^\mathsf{A}_{jk}(s) := \sqrt{\lambda_{jk}} K^\mathsf{A}_j (s\cdot \xi_{jk}) , 
  \quad K^\mathsf{A}_j (s\cdot \xi_{jk}) := \sum_{\ell = 0}^\infty b_j(\ell)  \sqrt{A_\ell} \frac{2\ell + 1}{4\pi} P_\ell (s\cdot \xi_{jk}) .
\end{equation}
{Note that due to the definition of $b_j$ in terms of $\kappa$ from Assumptions \ref{asskappa}, the summation over $\ell$ is finite.}
It is now straightforward to verify that \eqref{modneedletexpansion} holds for these functions.

\begin{theorem}\label{thm:modframe}
  Let Assumptions \ref{assLambda} and \ref{asskappa} hold.
  Then $\{{\psi}^\mathsf{A}_{j,k} \colon j\geq 0, k = 1,\ldots, n_j \}$ is a Parseval frame of the reproducing kernel Hilbert space of the Gaussian random field $u$ in \eqref{isotropicreal} with power spectrum $\mathsf{A}$ satisfying \eqref{Asummability}. 
  The functions $\psi^\mathsf{A}_{jk}$ are orthogonal in $L_2(\bbS^2)$ between non-adjacent levels $j$ and to polynomials of degree less than $\frac12 (2^{j-1} -1)$ when $j  > 1$.
\end{theorem}
  
\begin{proof}
 This follows immediately from Proposition \ref{prop:factorization}, since $\psi^\mathsf{A}_{jk} = S \psi_{jk}$ with $S$ as in \eqref{Sdef}, and the $\psi_{jk}$ form a tight frame of $L_2(\bbS^2)$ by Theorem \ref{thm:needletframe}.
 The further properties carry over directly from Theorem \ref{thm:needletframe}.
\end{proof}

As a consequence of Theorems \ref{thm:frame vs expansion} and \ref{thm:modframe}, we obtain the expansion \eqref{modneedletexpansion} with independent scalar coefficients of the isotropic Gaussian random field $u$ as in \eqref{isotropicreal} with power spectrum $\mathsf{A}$ such that \eqref{Asummability} holds.
We now turn to the more involved question under which conditions and to what extent the localisation properties \eqref{localizationestimate} of the standard needlets $\psi_{jk}$ are preserved in the modified needlets $\psi^\mathsf{A}_{jk}$.

\section{Localisation properties}
\label{sec:loc}

We now give a sufficient condition for the localisation properties analogous to \eqref{localizationestimate} of the modified needlets $\psi^\mathsf{A}_{jk}$ for the power spectrum $\mathsf{A}=(A_\ell)_{\ell\geq 0}$. It involves the forward differences of $\sqrt{A_\ell}$, which are defined recursively by
\begin{equation}\label{Aelldiff}
   \Delta_\ell^0 := \sqrt{A_\ell}, \qquad 
   \Delta_\ell^i := \Delta^{i-1}_{\ell+1} - \Delta^{i-1}_{\ell}
\end{equation}
for $i \in \N$ and $\ell \in \N_0$.

\begin{theorem}\label{modifieddecay}
 Let Assumptions \ref{assLambda} and \ref{asskappa} hold. In addition, let $\mathsf{A}=(A_\ell)_{\ell \in \N_0}$ satisfying \eqref{Asummability} be such that for some $\beta >0$ and $c_r>0$, 
 \begin{equation}\label{diffcondthm}
   \bigabs{\Delta_\ell^r } \leq c_r ( 1 + \ell)^{-(1+\beta + r) } \quad\text{for all $\ell \in \N_0$,}
 \end{equation}
 with $r$ as in Assumption \ref{asskappa}.
 Then for $j \in \N_0$, $k = 1,\ldots,n_j$, 
 \begin{equation}\label{modlocalizationestimate}
  \abs{\psi^\mathsf{A}_{jk}(s)} \leq  \frac{C 2^{-\beta j}}{1 + (2^j \theta)^r}
   \quad\text{for all $s \in \bbS^2$, where $\theta:=d (s, \xi_{jk}) \in [0,\pi]$,}
 \end{equation}
 with $C>0$ independent of $j$, $k$, and $s$.
\end{theorem}

\begin{remark}\label{rem:decayregularity}
As a consequence of Lemma \ref{lmm:alpha} below, if $\mathsf{A}$ satisfies \eqref{Asummability} and \eqref{diffcondthm}, then in particular there exists $C>0$ such that
\[
  A_\ell  \leq C ( 1 + \ell)^{- 2(1 + \beta)}, \quad \ell \geq 0.
\]
The summability condition in Remark \ref{rem:regularity} then applies, and thus the random field has a modification in $C^{\beta'}(\bbS^2)$ for any $\beta' < \beta$.
Conversely, \eqref{diffcondthm} is satisfied with any $r$ for $\mathsf{A}$ of power law type, for instance 
\begin{equation}\label{materndecay}
   A_\ell = ( 1 + \ell)^{- 2(1 + \beta)}, \quad \ell \geq 0. 
\end{equation}
The condition \eqref{diffcondthm} permits more general sequences $\mathsf{A}$, but restricts the asymptotics of their oscillations relative to such algebraic decay.
In the particular case of \eqref{materndecay}, one has \eqref{modlocalizationestimate} for any $r$, corresponding to decay faster than any polynomial in $\theta$.
\end{remark}

The proof of Theorem \ref{modifieddecay} requires some preparations.
Let  
\[
 \Z^* = \{ k + \textstyle\frac12 \displaystyle \colon k \in \Z \},
 \]
 and for given $\mathsf{A}$, define the sequence $( a_k )_{k \in \Z^*}$ by
\begin{equation}\label{adef}
  a_{\ell + \frac12} = \begin{cases} \sqrt{A_\ell}, \;& \ell \in \N \cup\{0\},\\ \sqrt{A_{-\ell-1}} , & -\ell \in \N.
 \end{cases}
\end{equation}
Note that $a_{\pm \frac12} = \sqrt{A_0}$, $a_{\pm \frac32} = \sqrt{A_1}$,  and so forth.

Since $\psi^\mathsf{A}_{jk}(s) = \sqrt{\lambda_{jk}} K_j^\mathsf{A}(\cos \theta)$ according to \eqref{modifiedneedlets}, we now derive an estimate for $K_j^\mathsf{A}(\cos \theta)$ with $\theta \in [0,\pi]$, for which we adapt the basic strategy from \cite{NPW}. For the Legendre polynomials, we have the Mehler-Dirichlet representation formula 
\[
P_\ell(\cos \theta) = \frac{\sqrt{2}}{\pi} \int_\theta^\pi \frac{\sin\left((\ell \plushalf) \varphi \right)}{\sqrt{\cos \theta - \cos \varphi}} d\varphi, 
\]
see, e.g., \cite[p.~85]{Sz}.
Hence with $\varepsilon_j = 2^{-j + 1}$, inserting $b_j(t) = \kappa(\varepsilon_j(t\plushalf))$, we have the integral representation
\begin{equation}\label{Ajdirichletmehler}
K^\mathsf{A}_j(\cos \theta) = \frac{\sqrt{2}}{8 \pi^2} \int_\theta^\pi \frac{ C^\mathsf{A}_{\varepsilon_j}(\varphi)}{\sqrt{\cos \theta - \cos \varphi}} d \varphi
\end{equation}
with
\[
  C^\mathsf{A}_\varepsilon(\varphi)
  = \sum_{\ell = 0}^\infty 
    \sqrt{A_\ell} \,\kappa\Bigl(\varepsilon \left(\ell \plushalf\right)\Bigr) \left(\ell \plushalf \right) \sin \Bigl( \left( \ell \plushalf\right)\varphi\Bigr).
\]

\begin{lemma}\label{lmm:Ctilde}
Let $\alpha \in C^r(\R)$, $r \geq 2$, be an even function satisfying $\alpha(k) = a_k$ for all $k \in \Z^*$. Then there exists $C>0$ depending only on $r$ and $\kappa$ such that for all $\varepsilon \in (0, \pi]$,
\begin{equation}\label{Ctildebound}
  \abs{ C^\mathsf{A}_\varepsilon(\varphi)} \leq C \biggl(   1 +  \biggl(\frac{\varphi}{\varepsilon} \biggr)^r  \biggr)^{-1} \max_{i=0,\ldots,r} \varepsilon^{-(i+1)}\norm{\alpha^{(i)}}_{L_1(\varepsilon^{-1} \supp \kappa)} , \quad \varphi \in [0,\pi] .
\end{equation}
\end{lemma}

\begin{proof}
Note first that for any fixed $\varphi > 0$, we can rewrite $C^\mathsf{A}_\varepsilon$ as
\begin{equation}\label{Ctildeexpansion}
   C^\mathsf{A}_\varepsilon(\varphi) = \frac12 \sum_{\mu\in \Z} g\!\left(\mu\plushalf\right), \quad g(t):= \kappa(\varepsilon t) \,\alpha(t) \, t\, \sin (t\varphi),
\end{equation}
and by the Poisson summation formula,
\begin{equation}\label{PSF}
  \sum_{\mu\in\Z} g\!\left( \mu \plushalf \right) = \sum_{\nu\in\Z} e^{\ii \pi \nu} \hat g(2\pi\nu).
\end{equation}
Using that $g$ is even and that $\kappa(\varepsilon t)\, \alpha(t)\, t$ is odd, we obtain
\begin{align*}
\hat{g}(\omega) &=  \int_{\mathbb{R}} \kappa(\varepsilon t)\, \alpha(t)\, t\, \sin (t \varphi)\, \cos (\omega t)\, dt \\
&= \frac{1}{2} \int_{\mathbb{R}} \kappa(\varepsilon t)\, \alpha(t)\, t\, \Bigl(\sin \bigl((\varphi + \omega)t \bigr) + \sin \bigl((\varphi - \omega)t\bigr) \Bigr)\, dt \\
&=  \frac{\ii}{2}  \int_{\mathbb{R}} \kappa(\varepsilon t)\, \alpha(t)\, t\, \bigl( e^{-\ii (\varphi + \omega)t } + e^{-\ii (\varphi - \omega)t} \bigr)\, dt \\
&= - \frac{1}{2 \varepsilon}   \int_{\mathbb{R}} \kappa(t) \,\alpha(t/\varepsilon ) (-\ii t/ \varepsilon) \bigl(e^{-\ii (\varphi + \omega)t/\varepsilon } + e^{-\ii (\varphi - \omega)t/\varepsilon} \bigr)\,dt \\
& = - \frac{1}{2 \varepsilon}   \frac{d}{d \varphi} \int_{\mathbb{R}} \kappa(t) \,\alpha(t/\varepsilon ) \bigl(e^{-\ii (\varphi + \omega)t/\varepsilon } + e^{-\ii (\varphi - \omega)t/\varepsilon}\bigr)\,dt .
\end{align*}
Thus we have
\begin{equation}\label{ghatdef}
  \hat g(\omega)  = - \frac{1}{2 \varepsilon} \frac{d}{d\varphi}  \bigl(   \rho(\varphi + \omega)  + \rho(\varphi - \omega) \bigr),  
\end{equation}
where
\[
  \rho(\varphi\pm\omega) := \mathcal{F}[\kappa \, \alpha(\cdot / \varepsilon)] \left( \frac{\varphi \pm \omega}{\varepsilon} \right), \quad \omega \in  \mathbb{R},\;\; r,j \geq 0.
\]

By elementary properties of the Fourier transform,
\[
  ( \ii \xi)^r \left( \ii \frac{d}{d\xi} \right)^j \mathcal{F}[f](\xi) = \mathcal{F} \left[t\mapsto \frac{d^r}{dt^r} \bigl(t^j f(t)\bigr) \right] (\xi) ,\quad \xi \in \R,
\]
for any $f$ having sufficient regularity and decay.
Thus for 
\[
  \zeta_\varepsilon(t) := \frac{d^r}{dt^r} \left( t \,\kappa(t)\, \alpha(t/\varepsilon)\right)
\]
we obtain
\[
  \left( \ii \frac{\varphi\pm \omega}{\varepsilon} \right)^r
    \left(  \ii\varepsilon \frac{d}{d\varphi}\right) 
      \rho(\varphi  \pm \omega) 
       = \widehat{ \zeta_\varepsilon } \left(\frac{\varphi\pm\omega}{\varepsilon}\right).
\]
As a consequence,
\begin{equation}\label{rhobound}
  \biggabs{ \left( \frac{\varphi\pm\omega}{\varepsilon} \right)^r \frac{d}{d\varphi} \rho(\varphi \pm \omega) } 
   \leq \frac{1}{\varepsilon} \norm{\zeta_\varepsilon}_{L_1}.
\end{equation}
Note that with a $C>0$ depending on $r$ and $\kappa$, but independent of $\varepsilon$,
\begin{equation}\label{zetabound}
\begin{aligned}
    \norm{\zeta_\varepsilon}_{L_1}
     &\leq \int_{\R} \Bigabs{\frac{d^r}{dt^r} \left( t \,\kappa(t)\, \alpha(t/\varepsilon)\right) } \,dt 
     \leq \int_{\supp \kappa} \sum_{i = 0}^r {r \choose i} \bigabs{ (t \kappa(t))^{(r-i)} (\alpha(t/\varepsilon))^{(i)} } \,dt \\
     &\leq C \max_{i = 0,\ldots,r} \norm{ \bigl(\alpha(\varepsilon^{-1}\cdot)\bigr)^{(i)}}_{L_1(\supp\kappa)} 
     = C  \max_{i = 0,\ldots,r} 
     	\varepsilon^{- i + 1}
       \norm{ \alpha^{(i)}}_{L_1(\varepsilon^{-1}\supp \kappa)}. 
 \end{aligned}
\end{equation}
Similarly, we find
\begin{equation}\label{rhozero}
   \biggabs{  \frac{d}{d\varphi} \rho(\varphi \pm \omega) } 
     \leq {C} \norm{ \alpha}_{L_1(\varepsilon^{-1}\supp \kappa)} \,.
\end{equation}

We now combine \eqref{rhobound} and \eqref{zetabound} and add \eqref{rhozero}, which by \eqref{ghatdef} yields 
\[
   \abs{\hat g(\omega)} \leq \frac{C}{2\varepsilon} 
   \left[
   \biggl(   1 +  \biggabs{\frac{\varphi + \omega}{\varepsilon} }^r  \biggr)^{-1} 
   + 
   \biggl(   1 +  \biggabs{\frac{\varphi - \omega}{\varepsilon} }^r  \biggr)^{-1} 
   \right]
   \max_{i=0,\ldots,r} \varepsilon^{-i}\norm{\alpha^{(i)}}_{L_1(\varepsilon^{-1} \supp \kappa)}  .
\]

Inserting this bound into \eqref{PSF} and using \eqref{Ctildeexpansion} gives
\[
   \abs{ C^\mathsf{A}_\varepsilon(\varphi)} \leq \frac{C}{2\varepsilon}   \sum_{\nu \in \Z}  \biggl(   1 +  \biggabs{\frac{\varphi + 2 \pi \nu}{\varepsilon} }^r  \biggr)^{-1}  \max_{i=0,\ldots,r} \varepsilon^{-i}\norm{\alpha^{(i)}}_{L_1(\varepsilon^{-1} \supp \kappa)}.
\] 
Proceeding similarly  as in \cite[Prop.~3.4]{NPW}, we can obtain a bound for the term 
\begin{equation}\label{seven}
   \sum_{\nu \in \Z}  \biggl(   1 +  \biggabs{\frac{\varphi + 2 \pi \nu}{\varepsilon} }^r  \biggr)^{-1}.
\end{equation}
For the convenience of the reader, we include this argument. First note that the dominant term of \eqref{seven} is obtained for $\nu = 0 $. Furthermore, since $\varphi \in [0,\pi]$, we have
\begin{align*}
 \sum_{\substack{\nu \in \Z \\ \nu \neq 0}}  \biggl(   1 +  \biggabs{\frac{\varphi + 2 \pi \nu}{\varepsilon} }^r  \biggr)^{-1} & \leq  
2  \sum_{\substack{\nu \in \Z \\ \nu \neq 0}} \left|\frac{2 \pi \nu -\pi}{\varepsilon}\right| ^{-r} \\&\leq
2 \left( \frac{\varepsilon}{\pi} \right)^r  \left( 1 + \int_1^\infty (2t - 1)^{-r} dt \right) \\&\leq 
2  \left( \frac{\varepsilon}{\pi} \right)^r  \frac{2r - 1}{r-1}.
\end{align*}
Next we multiply and divide on the right-hand side  by  $1 + (\frac{\varphi}{\varepsilon})^r$. Utilising one more time that  $\varphi \in [0,\pi]$, $r \geq 2$ and $\left( \frac{\varepsilon}{\pi} \right)^r  \leq 1$, we obtain
\begin{align*}
 \sum_{\substack{\nu \in \Z \\ \nu \neq 0}}  \biggl(   1 +  \biggabs{\frac{\varphi + 2 \pi \nu}{\varepsilon} }^r  \biggr)^{-1} & \leq  
 2  \frac{2r - 1}{r-1} \left(  \left( \frac{\varepsilon}{\pi} \right)^r  +1\right) \left(1 + \left(\frac{\varphi}{\varepsilon}\right)^r \right)^{-1} 
 \\& \leq 12 \left(1 + \left(\frac{\varphi}{\varepsilon}\right)^r \right)^{-1},
 \end{align*}
which together with the bound for $\nu =0$ gives \eqref{Ctildebound}.
\end{proof}

To further use \eqref{Ctildebound}, for given $(a_k)_{k \in \Z^*}$ as in \eqref{adef}, we need to construct functions $\alpha$ such that we have suitable bounds on
\begin{equation}
\max_{i=0,\ldots,r} \varepsilon_j^{-(i+1)}\norm{\alpha^{(i)}}_{L_1(\varepsilon_j^{-1} \supp \kappa)} , \quad \varepsilon_j = 2^{-j+1}, \; j \geq 0,
\end{equation}
and thus on the derivatives $\alpha^{(i)}$.
To this end, we now take a closer look at the connection between the $A_\ell$ and the asymptotic behaviour of possible choices of $\alpha$. 

\begin{lemma}\label{lmm:alpha}
Let $r \in \N_0$. The following two statements are equivalent:
\begin{enumerate}[{\rm (i)}]
\item
    The sequence $(A_{\ell})_{\ell \in \N_0}$ satisfies $\lim_{\ell \to \infty} A_\ell = 0$ and for some $\beta>0$ and $c_r>0$, 
    \begin{equation}\label{diffdecay}
         \bigabs{\Delta_\ell^r } 
        \leq c_r ( 1 + \ell)^{-(1 + \beta + r) }.
    \end{equation}
\item
    There exists an even function $\alpha \in C^r(\R)$ satisfying $\alpha(k) = a_k$ for all $k \in \Z^*$ with $a_k$ as in \eqref{adef} such that 
    \begin{equation}\label{derivdecay}
       \abs{\alpha^{(i)}(t)} \leq \bar c_r ( 1 + \abs{t})^{-(1 + \beta + i)},\quad t\in \R, \quad \text{for $i = 0,\ldots,r$,}
    \end{equation}
   where $\beta>0$ and $\bar c_r > 0$.
\end{enumerate}
\end{lemma}

\begin{proof}
Since the statement  for $r=0$ follows directly by piecewise linear, continuous spline interpolation, we can assume $r$ to be positive. 

We first show that (ii) follows from (i).
  Note that it suffices to show that for some $c>0$,
  \begin{equation}\label{alphacondition}
  \begin{aligned}
  & \abs{\alpha^{(r)}(t)} \leq  c ( 1 + \abs{t})^{-(1 + \beta + r)}  & & \text{for all $t \in \R$, and} \\
 &  \lim_{\abs{t}\to\infty} \alpha^{(i)}(t) = 0 ,   & & \text{for $i = 0, \ldots, r-1,$} 
  \end{aligned}
  \end{equation}
  since then for $t>0$,
  \[
  \begin{aligned}
   \abs{ \alpha^{(r-1)}(t) }  & \leq \int_t^\infty \abs{\alpha^{(r)}(\tau)} \,d\tau  \\
     &  \leq  c \int_t^\infty  ( 1 + \tau)^{-(1 + \beta + r)} \,d\tau =  \frac{c}{ \beta + r } ( 1 + t)^{-( \beta + r )} .
     \end{aligned}
  \]
  The same argument applies to $t<0$, 
  and by induction, the desired estimates for $i =  r-2, r-3,\ldots, 0$ then follow analogously.

  We now use cardinal spline interpolation to construct $\alpha$.
  Let $\mathcal{S}^*_n$ be the space of midpoint cardinal splines \cite[\S1]{S}, that is, any $S \in \mathcal{S}^*_n$ is $n$ times weakly differentiable and $S|_{(k,k+1)}$ is a polynomial of degree $n$ for each $k \in \Z^*$. As a consequence of \eqref{diffdecay}, by \cite[\S6, Thm.~1, Rem.~1]{S}, we can choose $\alpha$ as the unique $r$-times weakly differentiable midpoint cardinal spline in $\mathcal{S}^*_{2r-1}$   with $\alpha^{(r)} \in L_p(\R)$ for $1 \leq p < \infty$ such that $\alpha(k) = a_k$ for $k \in \Z^*$.

  In order to verify that this $\alpha$ satisfies \eqref{alphacondition}, for $k \in \Z^*$, let $M$ be the cardinal B-Spline with $2r+1$ knots $\{ -r, -r +1, \ldots, 0,\ldots, r-1, r \}$. Then as shown in \cite[\S4 Thm.~3 and \S4.5]{S}, $\alpha$ can be written as
  \[
     \alpha(t) = \sum_{k \in \Z^*} a_k \sum_{j\in \Z} w_j \,M(t - k - j),
  \]
  with a real-valued sequence $(w_j)_{j\in \Z}$ such that $w_j = w_{-j}$ for all $j$ and $\abs{w_j} \leq C \rho^j$ for some $C>0$ and $\rho \in (0,1)$. This shows in particular that $\alpha$ is even, that is, $\alpha(t) = \alpha(-t)$, $t \in \R$. Moreover, since $A_\ell \to 0$  as $\ell \to \infty$ and thus $a_k \to 0$ as $\abs{k}\to \infty$, we have $\alpha^{(i)}(t) \to 0$ as $\abs{t}\to \infty$ for $i = 0, \ldots, r$.

  In order to establish the first condition in \eqref{alphacondition}, we now proceed similarly to \cite[\S6.1]{S}.
  For all $k \in \Z^*$, we define the forward differences
  \[
  \begin{aligned}
   \Delta^i a_k &= \Delta^{i-1} a_{k+1} - \Delta^{i-1} a_k & &\text{with $\Delta^0 a_k = a_k$,}
   \\
   \Delta^i \alpha(k) &= \Delta^{i-1} \alpha(k+1) - \Delta^{i-1} \alpha(k) & &\text{with $\Delta^0 \alpha(k) = \alpha(k)$.}
   \end{aligned}
  \]
  By construction, $\Delta^i \alpha(k) = \Delta^i a_k$ for all $i, k$. 

  Since $\alpha^{(r)} \in \mathcal{S}^*_{r-1}$, there exists a representation 
  \begin{equation}\label{alphaderiv}
     \alpha^{(r)}(t) = \sum_{k \in \Z^*} c_k\, Q(t - k),
  \end{equation}
  where $Q$ is the B-Spline with knots $\{ 0, \ldots, r\}$. As a consequence of Peano's theorem (see \cite[\S2 eq.~(1.3)]{S}),
  \begin{equation}\label{peano}
    \Delta^r \alpha(\ell) = \int_\R \alpha^{(r)}(t)\, Q(t - \ell) \dt, \quad \ell \in \Z^*.
  \end{equation}
  Combining \eqref{alphaderiv} and \eqref{peano} and using the identity 
  \[
      \int_\R Q(t - k) \,Q(t - \ell) \dt = M(k - \ell), \quad k , \ell \in \Z^*,
  \]
  see \cite[\S6 eq.~(1.19)]{S}, after multiplying \eqref{alphaderiv} by $Q(t - \ell)$ and integrating we obtain
  \[
      \sum_{ k \in \Z^* } M( k - \ell ) \, c_k = \Delta^r a_\ell, \quad \ell \in \Z^*.
  \]
  As shown in \cite[\S4.5]{S}, this implies
  \[
     c_k = \sum_{j \in \Z^*} w_{k - j}\, \Delta^r a_j, \quad k \in \Z^*.
  \]
  Substituting this back into \eqref{alphaderiv} and using that for $j \in \Z^*$,
  \begin{equation}\label{diffrelation}
     \Delta^r a_j = \Delta^r_{j- \frac12} \quad\text{if $j>0$,}
      \qquad
      \Delta^r a_j = (-1)^r\Delta^r_{-j - r - \frac12}\quad \text{if $j < -r$,}
  \end{equation}
  yields the first condition in \eqref{alphacondition}. Together with $\alpha^{(i)}(t)\to 0$ as $\abs{t}\to \infty$, this now implies the estimates for $i = 0,\ldots, r-1$. 

To show that (ii) implies (i), we observe that again by  Peano's theorem,
\[
    \bigabs{\Delta^r \alpha(\ell)} \leq \int_\R \bigabs{\alpha^{(r)}(t)}\bigabs{ Q(t - \ell) } \dt
     \leq \max_{t \in [\ell, \ell + r]} \abs{ \alpha^{(r)}(t)},
\]
 where $Q$ is the B-Spline with knots $\{ 0, \ldots, r\}$. 
By \eqref{diffrelation}, this implies \eqref{diffdecay}. Moreover, $\lim_{\ell\to\infty} A_\ell = 0$ follows directly from \eqref{derivdecay} for $i=0$ and $A_k = A_{-k}$, for $k \in \Z^*$ is satisfied since  $\alpha$ is an even function.
\end{proof}

\begin{remark}
 As can be seen from the proof, the function $\alpha$ can be chosen as a cardinal spline interpolant of order $2r-1$ with knots in $\Z^*$.
\end{remark}

\begin{proof}[Proof of Theorem \ref{modifieddecay}.]
With $\alpha$ as in Lemma \ref{lmm:alpha}(ii), from \eqref{derivdecay} we obtain
\begin{equation}\label{alphaderivbound}
  \norm{ \alpha^{(i)}}_{L_1(\varepsilon^{-1}\supp \kappa)}
   \leq C_1 \varepsilon^{\beta + i}
\end{equation}
for $i = 0, \ldots, r$, with $C_1>0$ depending additionally on $\bar c_r$, but not on $\varepsilon$. Consequently, Lemma \ref{lmm:Ctilde} gives
\begin{equation}\label{Ctildedecay}
  \abs{ C^\mathsf{A}_\varepsilon(\varphi)} \leq \varepsilon^\beta \frac{C_2 \varepsilon^{-1}}{1 + \left( \frac{\varphi}{\varepsilon} \right)^r} 
\end{equation}
with a $C_2>0$.
Inserting this into \eqref{Ajdirichletmehler} and estimating the integral over $\varphi$ exactly as in \cite[Thm.~3.5]{NPW} (see also \cite[Thm.~13.1]{MP}) yields
\begin{equation}\label{Kjest}
  \abs{K_j^\mathsf{A}(\cos \theta)} 
    \leq  2^{-\beta j} \frac{ C_3 2^{j}}{1 +  (2^j \theta)^r} .
\end{equation}
Recall that $\psi^\mathsf{A}_{jk}(s) = \sqrt{\lambda_{jk}} K_j^\mathsf{A}(s\cdot \xi_{jk}) = \sqrt{\lambda_{jk}} K_j^\mathsf{A}(\cos\theta)$, where $\sqrt{\lambda_{jk}} \leq \sqrt{c} 2^{-j}$ by \eqref{lambdabound}. Combined with \eqref{Kjest}, we arrive at \eqref{modlocalizationestimate}.
\end{proof}

\begin{remark}
The type of assumptions on the power spectrum of the random field that we use to obtain localisation of expansion functions are closely related to those used in \cite{BCM} for stationary Gaussian random fields on $\R^d$ with covariance given by a function $k$ as in \eqref{eq:stationarycov}. There the Fourier transform $\hat k$ plays a very similar role as the power spectrum $\mathsf{A}$ in the present case. For the basis functions constructed there, localisation with algebraic decay of order $r>d$ is obtained if for some $t > d/2$, there exist $c,C>0$ such that $c ( 1+ \abs{\xi}^2)^{-t} \leq \hat k(\xi) \leq C ( 1 + \abs{\xi}^2)^{-t}$ and $\abs{\partial^\alpha \hat k (\xi)} \leq C ( 1 + \abs{\xi}^2)^{-(t + \abs{\alpha}/2)}$ for all $\xi \in \R^d$ for any multi-index $\alpha \in \N_0^d$ with $\abs{\alpha} \leq r$. For Mat\'ern covariances, these conditions are satisfied for any $r>0$, and one obtains superalgebraic decay, as in the case \eqref{materndecay} in the present setting.
\end{remark}

\begin{remark}
As Lemma \ref{lmm:alpha} shows, the assumptions of Theorem \ref{modifieddecay} yield the decay conditions \eqref{derivdecay} on $\alpha^{(i)}$ for $i=0,\ldots, r$ and are thus sufficient to ensure \eqref{alphaderivbound} for $i=0,\ldots, r$, which is subsequently used in Lemma \ref{lmm:Ctilde} for the proof of Theorem \ref{modifieddecay}. For this purpose, these assumptions are in fact also close to being necessary in the following sense: It is not difficult to see that $\norm{ \alpha^{(i)}}_{L_1(\varepsilon^{-1}\supp \kappa)}
   \leq C \varepsilon^{\beta + i }$ for $i=0,\ldots,r$ with a $C>0$ conversely already implies that there exists $\tilde C>0$ such that $\abs{\alpha^{(i)}(t)} \leq  \tilde C ( 1 + \abs{t})^{-(1 + \beta + i)}$ for $i = 0,\ldots, r-1$. At the expense of a substantially more technical argument, however, one could replace $\abs{\alpha^{(r)}(t)} \leq C ( 1 + \abs{t})^{-(1 + \beta + r)}$ by the weaker requirement $\norm{ \alpha^{(r)}}_{L_1(\varepsilon^{-1}\supp \kappa)}
   \leq C \varepsilon^{\beta + r }$ for a $C>0$.
\end{remark}

\section{{Numerical Implementation}}\label{sec:num}

\subsection{{Choice of quadrature and $\kappa$}}
The numerical realisation of the needlet representation of an isotropic Gaussian random field given by its power spectrum $(A_\ell)_{\ell \in \N_0}$ depends on a choice of quadrature on $\bbS^2$ satisfying Assumption \ref{assLambda} and a function $\kappa$ satisfying Assumption \ref{asskappa}. 
One possible choice for the quadrature points $\xi_{jk}$, $j \in \N_0$, $k = 1,\ldots, n_j$, are the spherical $t$-designs constructed by Womersley \cite{W}.
In this case, for each $j$, all weights are equal: $\lambda_{jk} = 4 \pi	/ n_j$ for $k = 1,\ldots,n_j$.

A possible choice of $\kappa$, similar to the construction of Meyer wavelets given in \cite{Dau}, is the following: let
\[
 \eta(x) = \frac{\eta_0(x) }{ \eta_0(x) + \eta_0( 1- x)}
 \quad\text{where ~}
  \eta_0(x) = \begin{cases} \exp( - x^{-1}), \quad & x > 0 ,\\
     0, & \text{otherwise,}
    \end{cases}
\]
and let
\begin{equation}\label{eq:kappaexample}
  \kappa(x) = \begin{cases} \sin \left( \frac{\pi}{2} \eta( 2x -1)\right) , \quad & x \in (\frac12, 1],\\
     \cos \left( \frac{\pi}{2} \eta ( x - 1) \right) , \quad & x \in (1, 2), \\
     0 , & \text{otherwise.}
      \end{cases}
\end{equation}
Then it is easy to see that $\kappa \in C^\infty$ and that $\kappa$ satisfies Assumption \ref{asskappa}. 

\subsection{{Numerical evaluation of modified needlets}}
For any given power spectrum, the task of numerically evaluating the modified needlets, which with the above choices are given by
\[
\begin{aligned}
  \psi^\mathsf{A}_{jk}(s) &= \sqrt{\lambda_{jk} } \sum_{\ell = 0}^\infty b_j(\ell)  \sqrt{A_\ell} \frac{2\ell + 1}{4\pi} P_\ell (s\cdot \xi_{jk})   \\
  & = \sqrt{ \frac{4 \pi}{n_j} }  \sum_{\ell = 0}^\infty \kappa\bigl( 2^{-j} (2 \ell + 1) \bigr)  \sqrt{A_\ell} \frac{2\ell + 1}{4\pi} P_\ell (s\cdot \xi_{jk}) 
  \end{aligned}
\]
is almost the same as the evaluation of standard needlets, which correspond to the special case $A_\ell = 1$, $\ell \geq 0$. Note that since $\supp \kappa = [\frac12, 2]$, the summation over $\ell$ effectively ranges over {$\ell = 2^{j-2},\ldots, 2^j - 1$}.

{We thus have $\psi^\mathsf{A}_{jk}(s) = R_j(s \cdot \xi_{jk})$ with} the $k$-independent radial components
\begin{equation}\label{Rjdef}
   R_j (t) :=  \sqrt{ \frac{4 \pi}{n_j} }  {\sum_{\ell = 2^{j-2}}^{2^j-1}} \kappa\bigl( 2^{-j} (2 \ell + 1) \bigr)  \sqrt{A_\ell} \frac{2\ell + 1}{4\pi} P_\ell (t),  \quad {t \in [-1,1]}.
\end{equation}
In this manner, one {obtains an explicit representation of $R_j$ as a polynomial of degree at most $2^j-1$ that can be directly used for evaluating $\psi^\mathsf{A}_{jk}(s) = R_j(s\cdot \xi_{jk})$ for any $k$ and $s$.
It can be evaluated numerically by any scheme for the fast evaluation of Legendre expansions. In particular, it can be converted at near-linear costs of order $\mathcal{O}(j^2 2^j)$ to a Chebyshev expansion by the scheme from \cite{HT}. This can subsequently be used to evaluate $R_j \circ \cos$ on any uniform grid of $2^J$ points in $[0,\pi)$ with $J\geq j$ by the fast discrete cosine transform, using $\mathcal{O}(J 2^J)$ operations.}

Figure \ref{fig:plots} shows numerical examples of the modified needlets with $A_\ell = (1 + \ell)^{-2(1 + \beta)}$, for $\beta=\frac12$ and $\beta = 2$, compared to the standard needlets $\psi_{jk}$. In view of Remark \ref{rem:decayregularity}, the assumptions of Theorem \ref{modifieddecay} are satisfied for any $r \in \N$ in this example. The plots were generated using an implementation of the scheme described above in Julia 1.5.2, where we used the  ApproxFun package \cite{OT} for evaluating Legendre expansions by transformation to Chebyshev expansions.

\begin{figure}[p]
\begin{tabular}{ccc}
 \includegraphics[height=3.2cm]{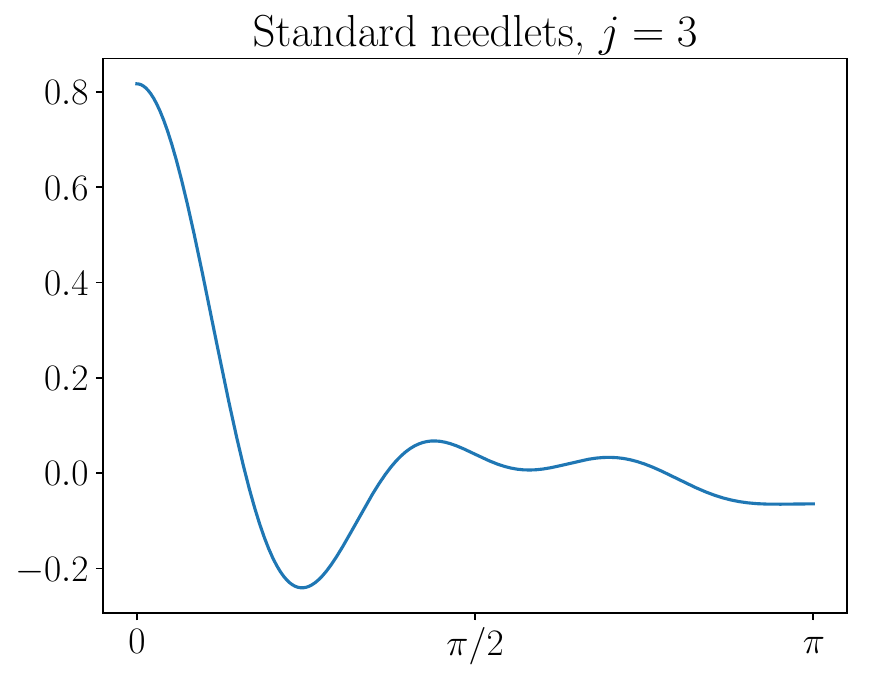}  & \includegraphics[height=3.25cm]{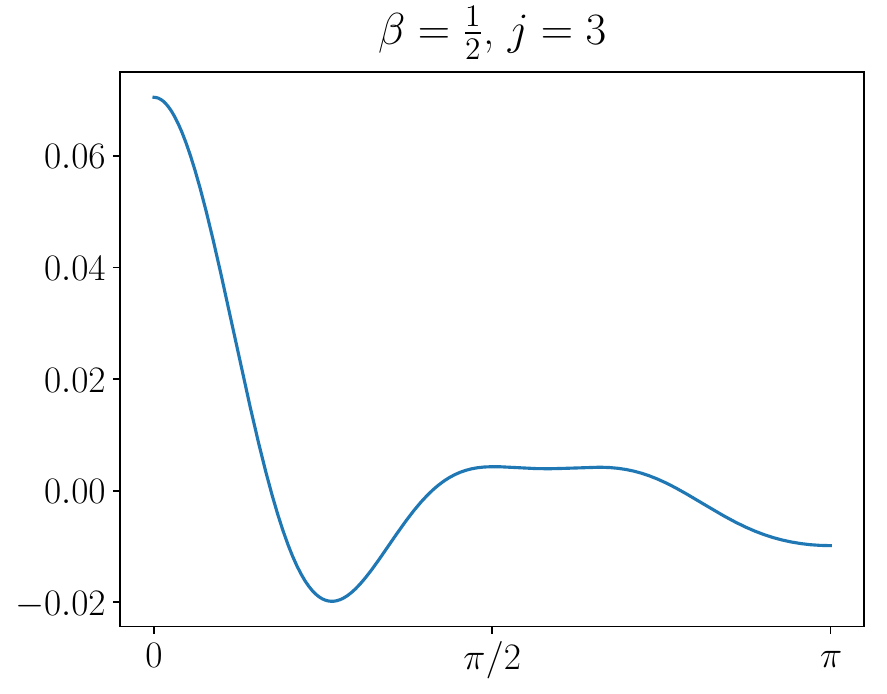} & \includegraphics[height=3.2cm]{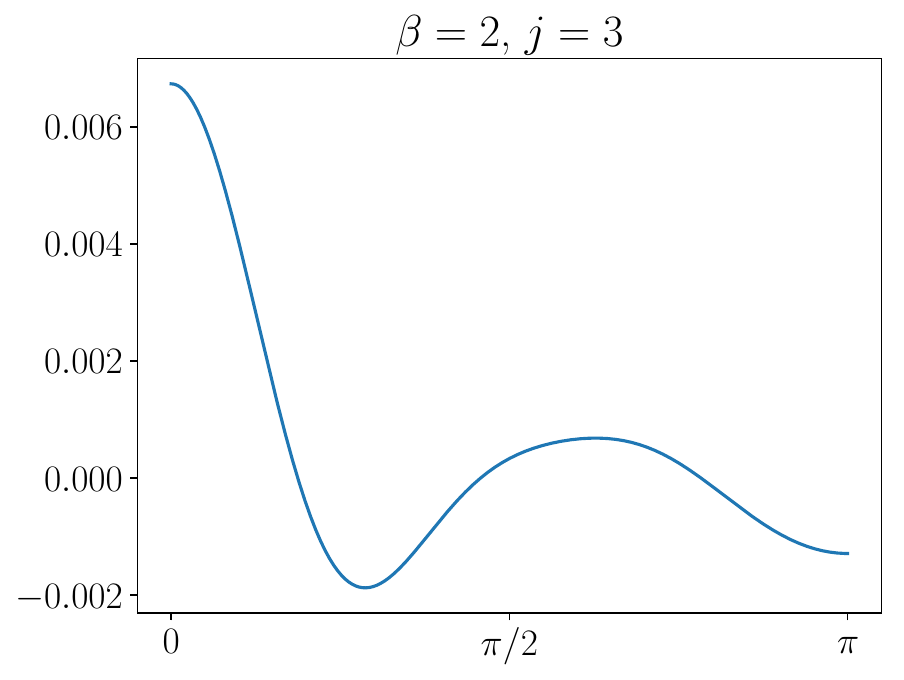} \\[-6pt]
 \includegraphics[width=4cm]{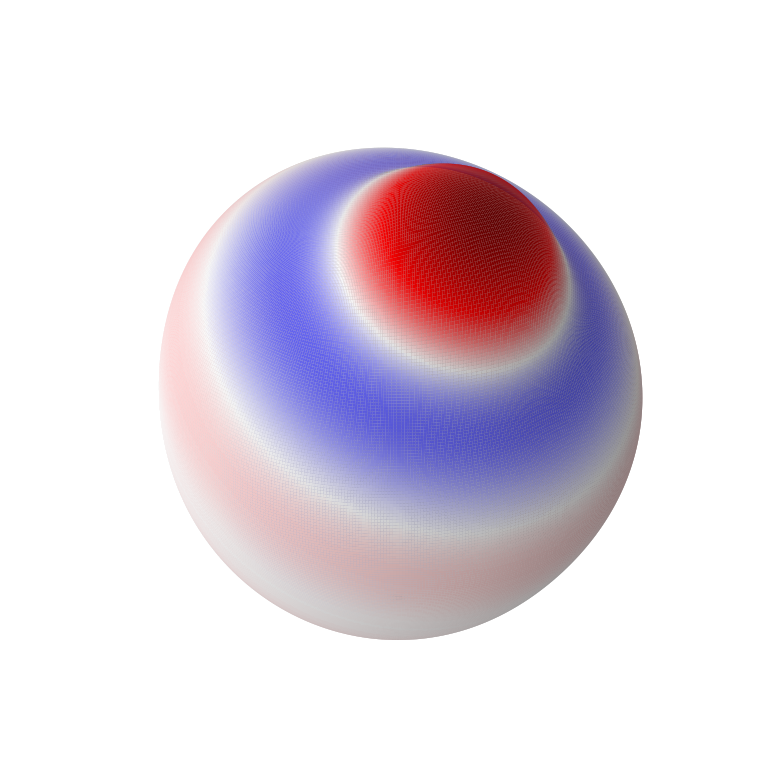} &  \includegraphics[width=4cm]{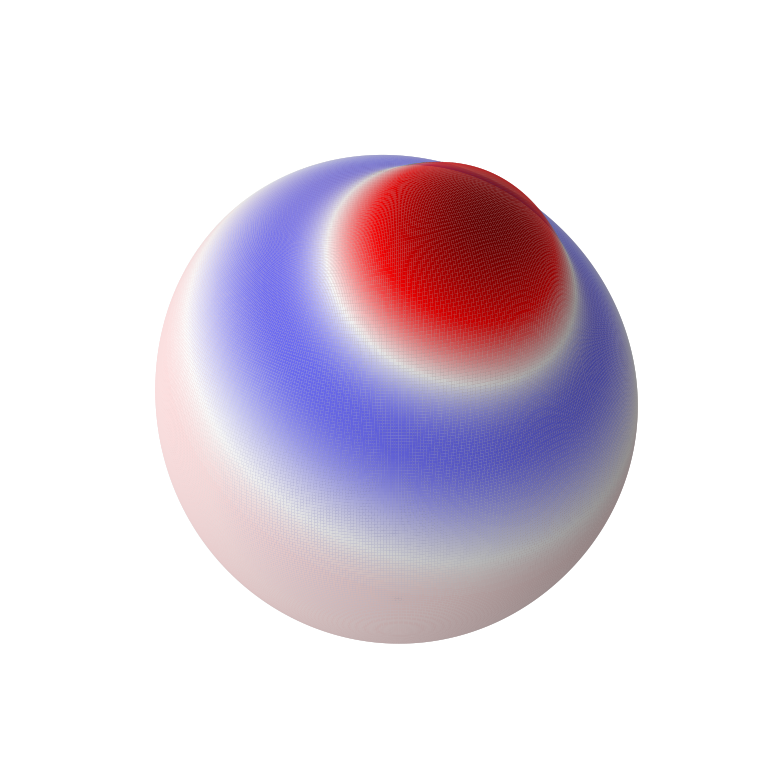} &  \includegraphics[width=4cm]{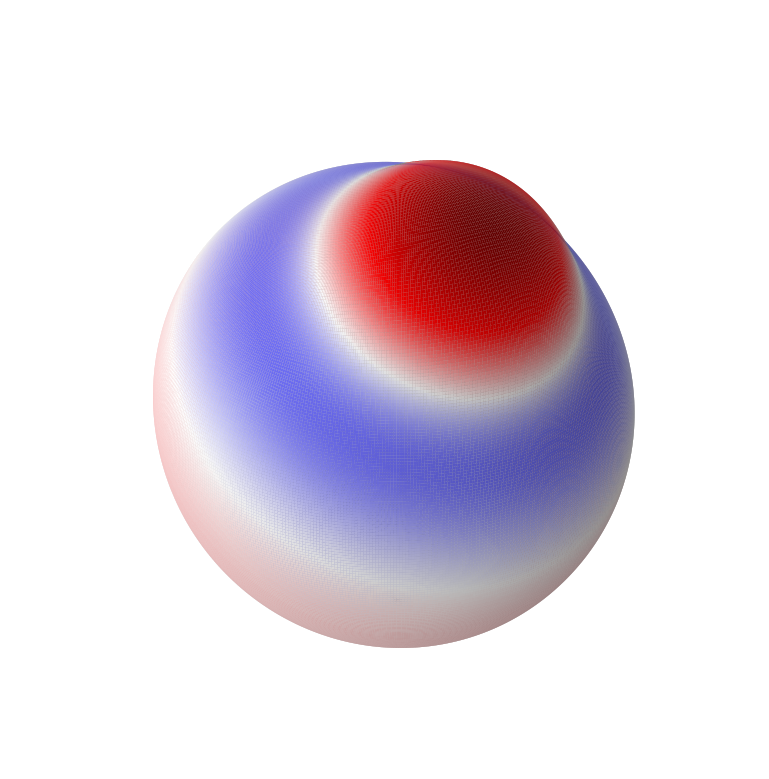} \\[3pt]
 \includegraphics[height=3.2cm]{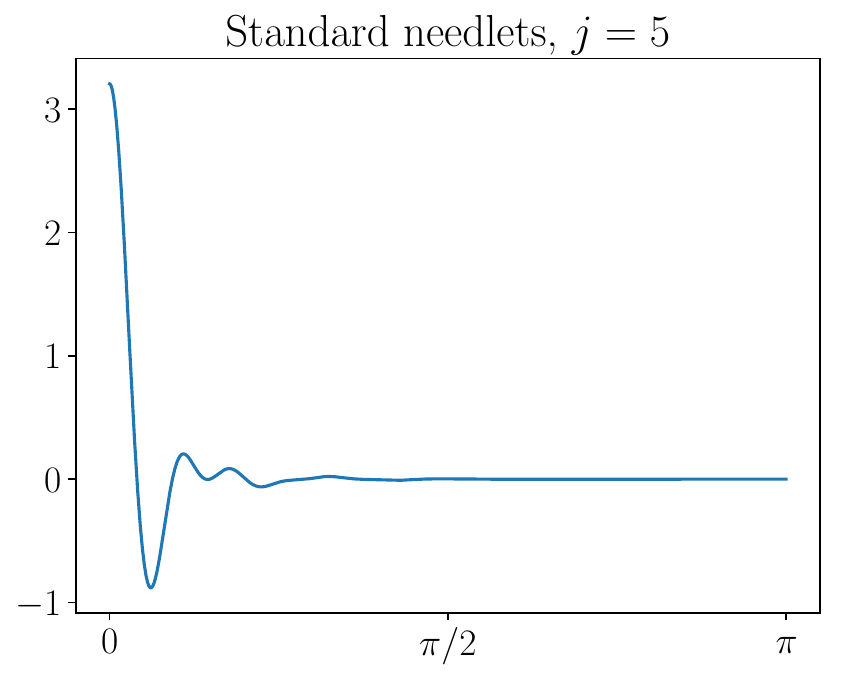}  & \includegraphics[height=3.25cm]{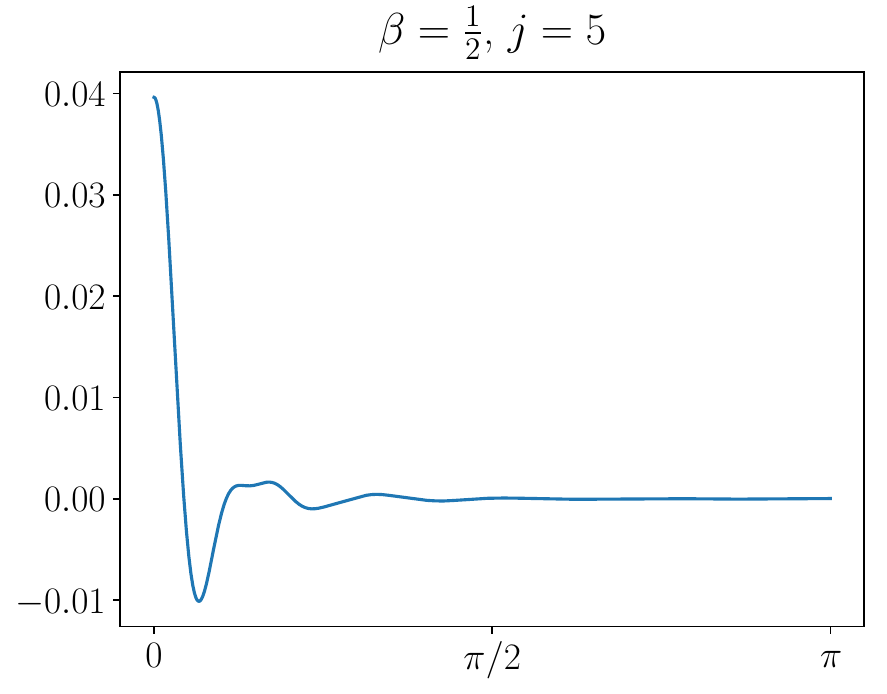} & \includegraphics[height=3.2cm]{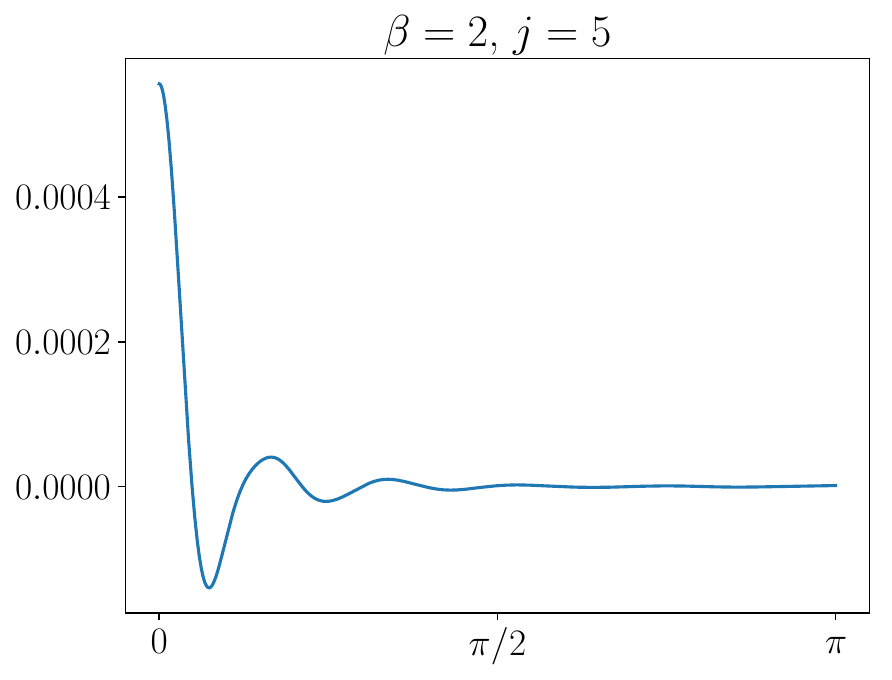} \\[-6pt]
 \includegraphics[width=4cm]{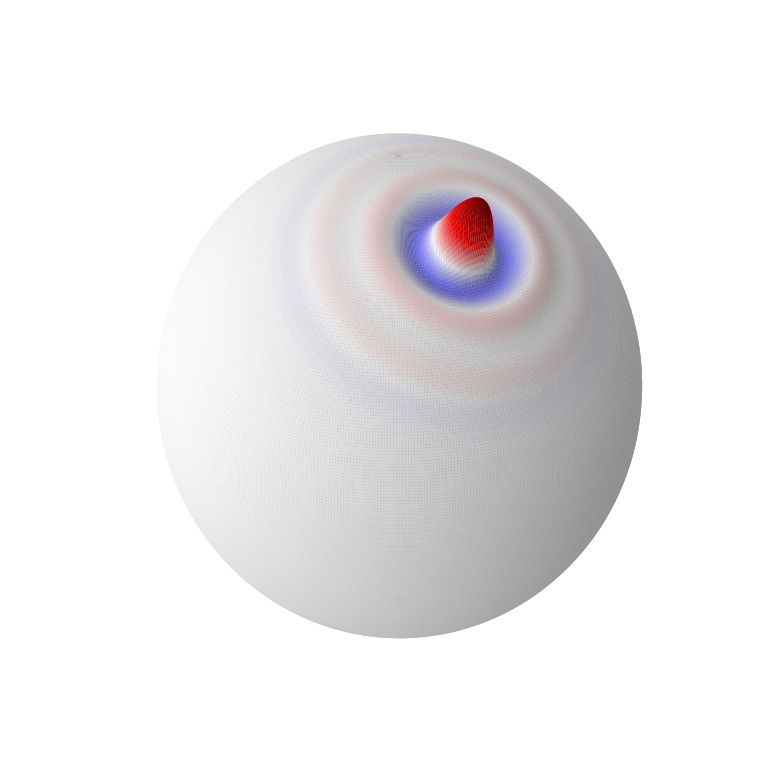} &  \includegraphics[width=4cm]{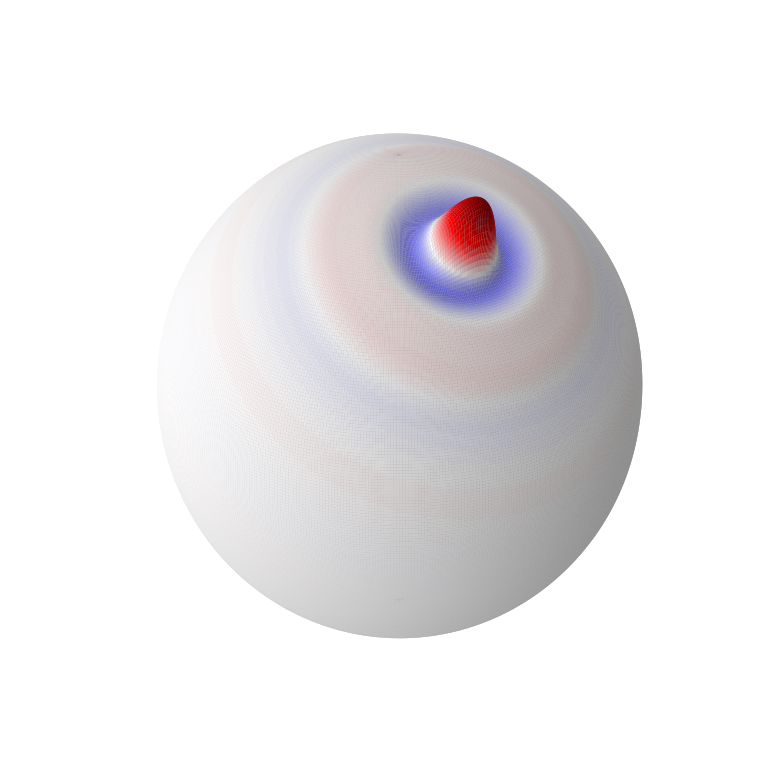} &  \includegraphics[width=4cm]{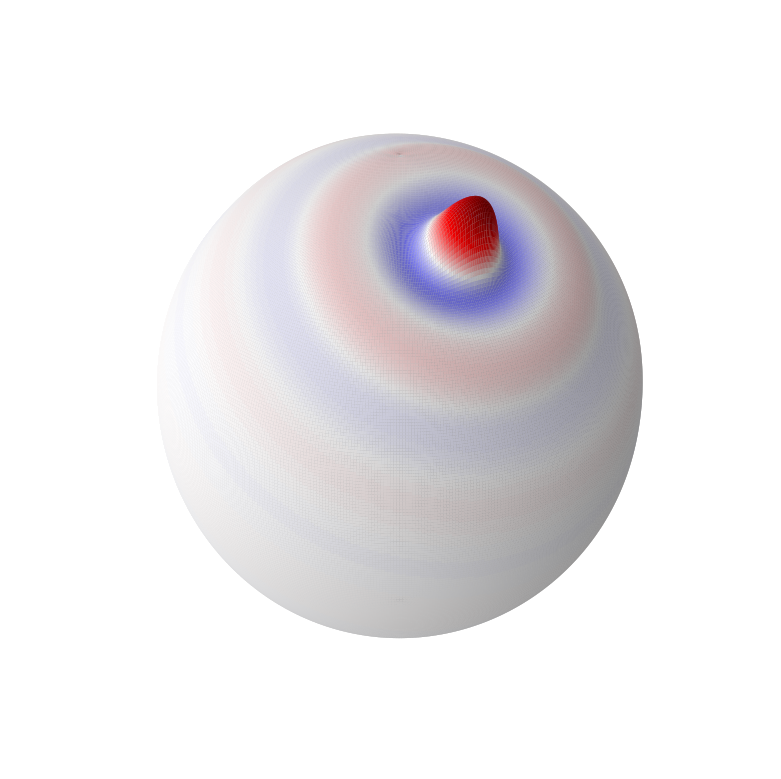} \\[3pt]
 \includegraphics[height=3.2cm]{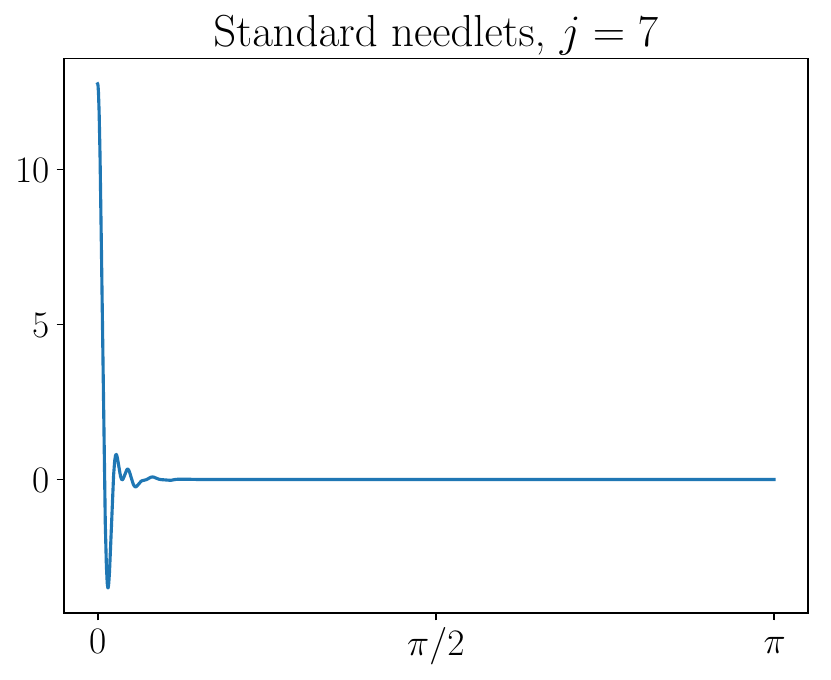}  & \includegraphics[height=3.25cm]{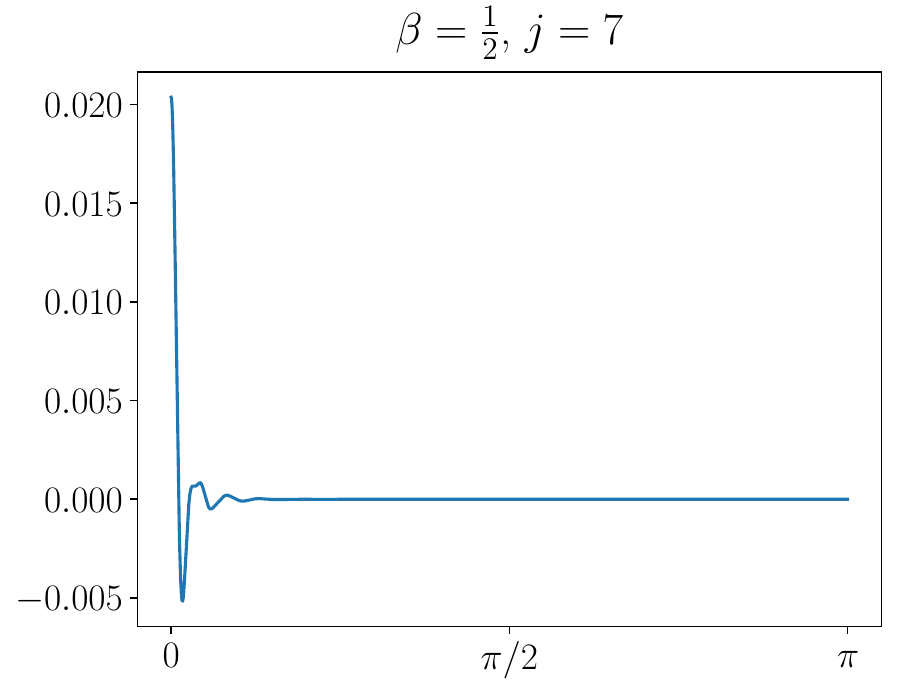} & \includegraphics[height=3.2cm]{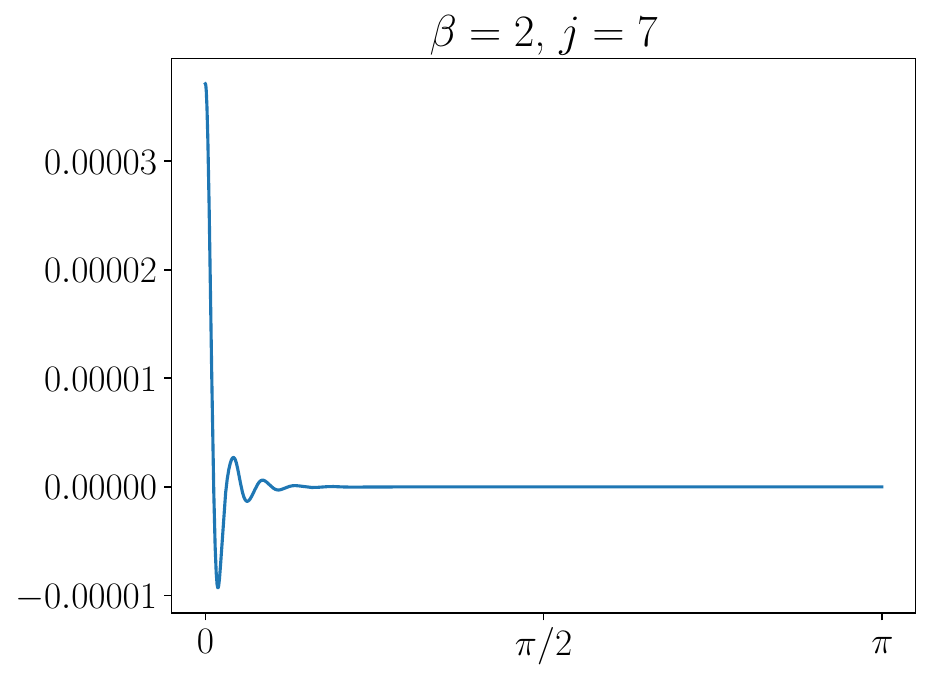} \\[-6pt]
 \includegraphics[width=4cm]{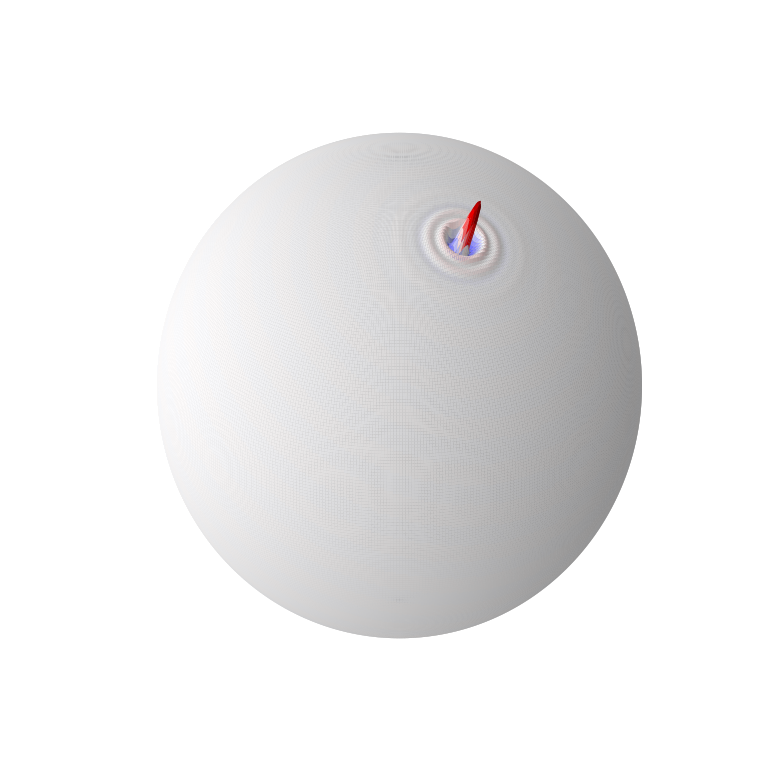} &  \includegraphics[width=4cm]{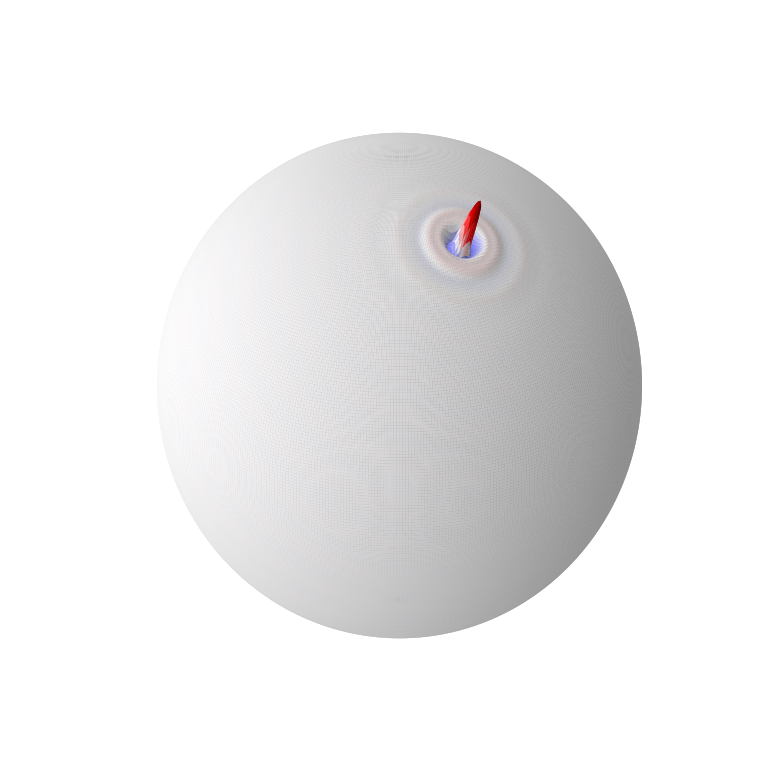} &  \includegraphics[width=4cm]{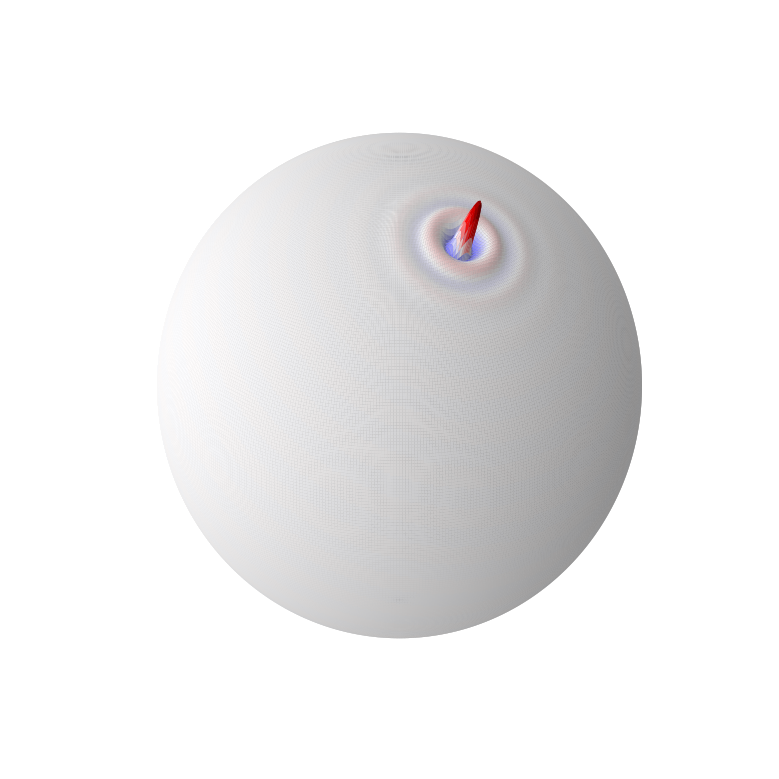} 
\end{tabular}
\caption{Standard needlets (left column) and modified needlets for $A_\ell = (1 + \ell)^{-2(1 + \beta)}$ with $\beta=\frac12$ (middle column) and $\beta = 2$ (right column), each for $j=3,5,7$; shown are the angular component $R_j$ and the corresponding functions on $\bbS^2$. }\label{fig:plots}
\end{figure}

\subsection{{Spline approximations}}\label{sec:splinetrunc}

{The computational costs of evaluating $\psi^\mathsf{A}_{jk}(s)$ for given $s$ can be reduced by approximating the functions $R_j \circ \cos$ by spline interpolants on sufficiently fine grids. The required size of such grids can be quantified based on the following bounds on higher derivatives of these functions.}

\begin{theorem}\label{thm:needletderiv}
{Let Assumptions \ref{assLambda} and \ref{asskappa} hold and for some $\beta>0$, assume that there exists $C>0$ such that $\mathsf{A}=(A_\ell)_{\ell \in \N_0}$ satisfies $A_\ell  \leq C ( 1 + \ell)^{- 2(1 + \beta)}$, $\ell \geq 0$. Then for each $n \in \N_0$, there exists $C>0$ such that
\begin{equation}
   \bigabs{ R_j }_{C^n([-1,1])} \leq C 2^{(2n-\beta)j}, \qquad
    \bigabs{ R_j \circ \cos}_{C^n(\R/2\pi\Z)} \leq C 2^{(n-\beta) j} 
\end{equation}
for $j \in \N_0$, $k = 1,\ldots, n_j$.}
\end{theorem}

\begin{proof}
  {For the ultraspherical polynomials $P^{(\lambda)}_\ell$, we have (see \cite[(4.7.14)]{Sz}) the identity
  \begin{equation}\label{eq:ultraspherical1}  \frac{d}{dt} P^{(\lambda)}_\ell(t) = 2\lambda P^{(\lambda+1)}_{\ell-1}(t)\,,  \end{equation}
  where $P^{(\frac12)}_\ell = P_\ell$. As noted in \cite[Sec.~3.1]{NPW},
  \begin{equation}\label{eq:ultraspherical2} \bigabs{P^{(\lambda)}_\ell(t)} \leq P^{(\lambda)}_\ell(1) = {\ell + 2\lambda - 1\choose \ell} \leq (1 + \ell)^{2\lambda-1} \,. \end{equation}
  Hence, for $\ell \geq n$ we have
  \begin{equation}
\left|\frac{d^n}{dt^n} P_\ell(t)\right| \leq  \Bigl|P^{(\frac{1}{2}+n)}_{\ell-n}(t)\Bigr| \leq (1+\ell-n)^{2n} \,.
\end{equation}
  With our assumption on $A_\ell$ and $n_j$, we thus obtain
  \[ \biggabs{ \frac{d^n}{dt^n}   R_j(t) } \leq C 2^{-j} \sum_{\ell = 2^{j-2}}^{2^j-1}  (1 + \ell)^{-(1+\beta)}(1 + \ell)^{1 + 2n}
   \leq C 2^{(2n - \beta) j} \,. \]
  As a consequence of Fa\`a di Bruno's formula, ${d^n R_j(\cos\theta)}/{d\theta^n}$ can be written as a sum of terms of the form \[ c_{n_1,n_2} \frac{d^{n_1+n_2} R_j(\cos\theta)}{d\theta^{n_1+n_2}} (\sin^{n_1}\theta)( \cos^{n_2} \theta), \] where $n_1 + 2n_2 \leq n$ and where both the coefficients $c_{n_1,n_2}$ and the number of summands satisfy bounds depending only on $n$. Combining Bernstein's inequality for higher derivatives (see \cite[Sec.~5.2.E.5]{BE95}) with \eqref{eq:ultraspherical1} and \eqref{eq:ultraspherical2}, we obtain
  \[ 
    \begin{aligned}
    \left| \frac{d^{n_1+n_2}}{dt^{n_1+n_2}} P_\ell(t) \right|  
    & \leq 2^{n_2} \Bigl( \prod_{i=0}^{n_2-1} (\textstyle\frac12 + i\displaystyle) \Bigr) \left|  \frac{d^{n_1}}{dt^{n_1}}  P^{(\frac12 + n_2)}_{\ell-n_2}(t)  \right| \\
    & \leq  2^{n_2} \Bigl( \prod_{i=0}^{n_2-1} (\textstyle\frac12 + i\displaystyle) \Bigr) \left(\frac{2\ell}{\sqrt{1-t^2}}\right)^{n_1} \sup_{t\in[-1,1]} \bigabs{ P^{(\frac12 + n_2)}_{\ell-n_2}(t) }  \leq  \frac{C (1+\ell)^{n_1 + 2 n_2}}{(\sqrt{1-t^2})^{n_1}} \,, 
    \end{aligned}\]
  and thus using $\sqrt{1 - \cos^2 \theta} = \sin \theta$ as well as \eqref{lambdabound} gives
  \[ \biggabs{ \frac{d^n}{d\theta^n}   R_j(\cos\theta) } \leq C 2^{-j} \sum_{\ell = 2^{j-2}}^{2^j-1}  (1 + \ell)^{-(1+\beta)} (1 + \ell)^{1+n_1+2n_2} \leq C 2^{(n - \beta)j} \,. \qedhere  \]
  }
\end{proof}

{Let $\mathcal{S}_{n,\tau}$ denote the space of continuous $2\pi$-periodic spline functions on $[0,2\pi)$ of piecewise polynomial degree $n \in \N$ with grid spacing $\tau>0$. Then as a consequence of Theorem \ref{thm:needletderiv}, 
\begin{equation}\label{eq:splineapprox}
  \min_{S \in \mathcal{S}_{n,\tau}} \bignorm{ (R_j\circ\cos) - S}_{C(\R/2\pi\Z)} \leq C \tau^{n+1} 2^{(n+1-\beta)j}\,.
\end{equation}
For reducing the cardinality of such grids, in particular for large $j$, the localisation of the needlets can be exploited in a further approximation, replacing $R_j$ by radially truncated functions $R_{j,c} = \Chi_{[\cos c, 1]}R_j$ with suitable $c \in (0,\pi]$.
By \eqref{modlocalizationestimate}, for any $\varepsilon >0$, we obtain $\norm{ R_j - R_{j,c} }_{C([-1,1])} \leq \varepsilon$ provided that
\begin{equation}\label{eq:truncsupp}
   c \geq C^{\frac1r} \varepsilon^{-\frac1r} 2^{- (\frac{\beta}r + 1)j} \,
\end{equation}
where $C$ may depend on $\beta$ and $r$.}

\section{{Applications}}

\subsection{{Random PDEs on $\bbS^2$}}\label{sec:pde}

In problems of uncertainty quantification, one is often interested in partial differential equations with random fields as coefficients. Series expansions of these random fields can be used in the construction of deterministic approximations of corresponding random solutions. 
We now give an application of the expansions \eqref{modneedletexpansion} to elliptic PDEs on the sphere with lognormal coefficients. {The numerical analysis of of lognormal diffusion problems on the sphere using the KL expansion has been considered also in \cite{HLS18}.}
Here we adapt the results for analogous problems on bounded domains in \cite{BCDM}. To simplify notation, we set $\Gamma := \bbS^2$.
We consider the diffusion problem
\begin{equation}\label{eq:Poisson}
- \nabla_\Gamma \cdot (a \nabla_\Gamma q) = f
\end{equation}
in the following weak formulation: for given $a \in L_\infty(\Gamma)$ with $a>0$ a.e.\ such that $a^{-1} \in L_\infty(\Gamma)$ as well as
\[
  f \in   L_{2,0}(\Gamma) := \biggl\{  f\in L_2(\Gamma) \colon \int_{\Gamma} f \dS = 0   \biggr\},
\]
find $q \in V := H^1(\Gamma) \cap L_{2,0}(\Gamma)$ such that
\begin{equation}\label{eq:poissonweak}
   \int_{\Gamma} a \nabla_{\Gamma} q \cdot \nabla_{\Gamma} \phi \dS = \int_{\Gamma} f\,\phi  \dS , \quad
    \phi \in V.
\end{equation}
By Poincare's inequality on $\Gamma$ (see \cite{DE}),
\begin{equation}\label{apriori}
\|  q \|_{L_2(\Gamma)} \leq C_\mathrm{P} \| \nabla_{\Gamma} q \|_{L_2(\Gamma)},  \quad  q \in V,
\end{equation}
with a $C_{\mathrm{P}} >0$,
and by the Lax-Milgram theorem
we obtain the existence and uniqueness of the weak solution $q$ and the a priori bound
\begin{equation}\label{solutionbound}
   \norm{q}_{V} \leq C_\mathrm{P} \norm{ a^{-1} }_{L^\infty(\Gamma)} \norm{f}_{L^2(\Gamma)}.
\end{equation}

We consider lognormal random coefficients $a = \exp( u )$, where $u$ is an isotropic Gaussian random field on $\bbS^2$. By Theorem \ref{thm:modframe}, we have the representation
\begin{equation}\label{modneedletexpansion2}
   u = \sum_{(j, k) \in  \mathcal{J}} y_{j, k}  \psi^\mathsf{A}_{jk}
\end{equation}
with i.i.d.\ coefficients $y_{j,k} \sim \cN (0,1)$, where $\psi^\mathsf{A}_{jk}$ are defined as in \eqref{modifiedneedlets} and where $\mathsf{A}$ is the power spectrum of $u$. The aim is now the efficient approximation of the $V$-valued random variable defined by the weak solution $q$ for each realisation of the scalar random coefficients $y_{j,k}$, 
\[
   y \mapsto q(y) \in V, \quad y = (y_{j,k})_{(j,k)\in\mathcal{J} }.
\]
Here the random vector $y$ is distributed according to the product measure 
\[
  \gamma := \bigotimes_{ (j,k)\in\mathcal{J}} \mathcal{N}(0,1)
\]
on $U := \R^\N$.

The further analysis of integrability and approximability of the random solutions $q(y)$ hinges on summability properties of the functions $\psi_{jk}^\mathsf{A}$.
Using the localisation estimate from Theorem \ref{modifieddecay}, we now verify such a summability condition on the sphere that is analogous to the condition \cite[eq.~(16)]{BCM} on a bounded domain.

\begin{cor}\label{psisumdecay}
Under the assumptions of Theorem \ref{modifieddecay}, there exists $C>0$ such that for each $j$,
\begin{equation}\label{uniformPoints}
    \sup_{s \in \bbS^2} \sum_{k = 1}^{n_j} \abs{\psi^\mathsf{A}_{jk}(s)} \leq C 2^{- \beta j}\,.
\end{equation}
\end{cor}

\begin{proof}
For fixed $s \in \bbS^2$ and $j \in \N_0$, we define the sets  
\begin{equation}\label{Xindef}
  \Xi_n = \{  k \colon  (n-1) h_j \leq d(s,\xi_{jk}) \leq n h_j  \}, \quad n \in \N.
\end{equation}
While these are empty for sufficiently large $n$, for our purposes it suffices to observe that due to \eqref{meshratio}, we have
\[
  \# \Xi_n \leq c n , \quad n \in \N,
\]
with $c>0$ independent of $s$ and $j$.
With \eqref{modlocalizationestimate}, we obtain
\[
\begin{aligned}
   \sum_{k = 1}^{n_j} \abs{\psi^\mathsf{A}_{jk}(s)}
    & \leq  C 2^{- \beta j} \sum_{n \in \N} \sum_{k \in \Xi_n} \frac{1}{1 + \bigl(2^j d(s,\xi_{jk})\bigr)^r}  \\
    & \leq  c C 2^{-\beta  j} \sum_{n \in \N} \frac{n}{ 1 + \bigl(2^j h_j (n-1)\bigr)^r} .
 \end{aligned}
\]
By \eqref{hdecrease}, there exists $c_0>0$ such that $2^j h_j \geq c_0$ for all $j$, and since $r>2$ by Assumption \ref{asskappa},
\[
  \sum_{n \in \N} \frac{n}{ 1 + \bigl(2^j h_j (n-1)\bigr)^r} \leq \sum_{n \in \N} \frac{n}{ 1 + \bigl(c_0(n-1)\bigr)^r} < \infty,
\]
completing the proof.
\end{proof}

From \eqref{solutionbound}, we obtain
\[
\| q(y)\|_{V} \leq C_\mathrm{P} \| f\|_{L^2(\Gamma)} \exp \biggl(  \Bignorm{ \sum_{j, k} y_{j, k} \psi^\mathsf{A}_{jk} }_{L_\infty(\Gamma)} \biggr),
\]
provided that the right hand side is defined for the given $y$.
From \eqref{uniformPoints}, proceeding exactly as in \cite[Cor.~2.3]{BCDM}, one obtains that a solution $q \in V$ exists for $\gamma$-almost every $y$ and moreover, $q \in L_p(U, V, \gamma)$ for any $p<\infty$. 

Since in particular $q \in L_2(U, V, \gamma)$, we have an expansion of $q$ in terms of product Hermite polynomials.
Let $\cF$ be the set of finitely supported sequences of non-negative integers $\nu=(\nu_{j,k})_{(j,k) \in \mathcal{J}} \in \N_0^\mathcal{J}$ and let $(H_n)_{n\geq 0}$ be the sequence of univariate Hermite polynomials normalised with respect to the density of $\mathcal{N}(0,1)$.
The expansion of $q$ with respect to the orthonormal basis $\{ H_\nu \}_{\nu \in \cF}$ with 
$H_\nu(y) := \prod_{(j,k) \in \mathcal{J}} H_{\nu_{j,k}}(y_{j,k})$ of $L_2(U,\gamma)$ then reads
\begin{equation}
q(y) = \sum_{\nu \in \cF} q_\nu H_\nu(y) ,\qquad q_\nu=\int_U q(y)\,H_\nu(y)\sdd\gamma(y) \in V ,
\end{equation}
with convergence in $L_2(U, V, \gamma)$.

Let $(\nu^*_k)_{k\in\N}$ be an enumeration of $\cF$ such that $\norm{ q_{\nu^*_1} }_V \geq \norm{ q_{\nu^*_2} }_V \geq \ldots$, and for each $n \in \N$, let $\Lambda_n = \{ \nu^*_1,\ldots,\nu^*_n \}$. With this index set corresponding to $n$ largest values of $\norm{ q_\nu }_V$, a best $n$-term approximation of $q$ is given by
\begin{equation}\label{bestnterm}
   q_n(y)  :=  \sum_{\nu \in \Lambda_n} q_\nu H_\nu(y) .
\end{equation}

Under the levelwise decay condition \eqref{uniformPoints}, we apply on $\Gamma$ exactly the same steps as carried out in \cite{BCDM} on bounded domains (see also \cite{BCM} for a summary) to arrive at the following convergence result for best $n$-term approximations based on the expansion \eqref{modneedletexpansion2} of the random field $u$.

\begin{theorem}\label{thm:hermiteconv}
 Under the assumptions of Theorem \ref{modifieddecay}, for each $s \in ( 0 ,\beta /2)$ there exists $C_{q,s} >0$ such that the $q_n$ as in \eqref{bestnterm} satisfy 
 \[
    \norm{ q - q_n }_{L_2(U,V,\gamma)} \leq C_{q,s} n^{-s}.
 \]
\end{theorem}

In other words, the best $n$-term product Hermite polynomial approximations $q_n$ converge in $L_2(U,V,\gamma)$ as $\mathcal{O}(n^{-s})$ for any $s$ up to $\beta/2$. As noted in Remark \ref{rem:decayregularity}, under the assumptions of Theorem \ref{modifieddecay}, $\beta$ is precisely the limiting order of H\"older regularity of the random field $u$ (and hence of $a$), that is, there exists a modification such that realisations are in $C^{\beta'}(\Gamma)$ for any $\beta' < \beta$.
{Note that here we rely crucially on the localisation of the random field expansion expressed in \eqref{uniformPoints}, and no such convergence result is available for KL expansions in spherical harmonics.
Noting that $\Gamma$ has dimensionality $2$, our new result parallels the one of \cite{BCM} for domains $D\subset \R^d$:} for random fields of H\"older regularity up to order $\beta$, the wavelet expansions constructed there yield $n$-term Hermite approximations of solutions that yield an error of order $\mathcal{O}(n^{-s})$ for any $s \in (0, \beta /d)$ in $L_2(U, H^1_0(D), \gamma)$.

\subsection{{Convergence of truncated expansions and sampling of random fields}}
{Under the assumptions of Theorem \ref{modifieddecay}, for an isotropic Gaussian random field $u$ with power spectrum $\mathsf{A}$, one may also be interested in a specific type of convergence of the partial random series 
\begin{equation}\label{partsum}
   u_J(y) = \sum_{\substack{(j, k) \in  \mathcal{J} \\ j \leq J}} y_{j, k}  \psi^\mathsf{A}_{jk}\,.
\end{equation} 
As we have noted, truncation of the KL expansion \eqref{genKL} in particular yields the fastest convergence with respect to the number of terms in the mean-squared sense, that is, in $L_2(U,L_2(\bbS^2),\gamma)$. As we now show, the alternative expansions with localization as in \eqref{partsum} are especially suitable for obtaining convergence in $C(\bbS^2)$.}

{Based on the estimate \eqref{uniformPoints}, one can estimate $\norm{ u(y) - u_J(y)}_{C(\bbS^2)}$ in various ways; we now illustrate this for convergence in $L_p(U,C(\bbS^2),\gamma)$ for $p\in(0,\infty)$.
Following \cite[Thm.~2.2]{BCDM}, it is not difficult to see that for any sequence of positive real numbers $(\omega_{j,k})_{(j,k)\in \mathcal{J}}$ such that $\sum_{(j,k) \in \mathcal{J}} \exp(-\omega_{j,k}^2) < \infty$, we have 
\[ \int_U  \bigl( \sup_{(j,k)\in\mathcal{J}} \omega_{j,k}^{-1} \abs{y_{j,k}} \bigr)^p \sdd \gamma(y) < \infty\,. \]
With $\beta>0$ as in \eqref{uniformPoints}, let $\omega_{j,k} = \sqrt{3 \ln 2^{j+1}}$. Then by \eqref{uniformPoints},
\begin{equation}\label{uJest}
\begin{aligned}
   \norm{u(y) - u_J(y)}_{C(\bbS^2)} &\leq  \biggl( \sup_{j>J}  2^{\beta j} \Bignorm{ \sum_{k=1}^{n_j} \bigabs{\psi^\mathsf{A}_{jk}}}_{C(\bbS^2)} \biggr)\biggl( \sum_{ j > J } 2^{-\beta j} \sup_{k=1,\ldots,n_j} \abs{y_{j,k}} \biggr) \\
   &\leq C  \Bigl( \sup_{(j,k) \in \mathcal{J}} \omega_{j,k}^{-1} \abs{y_{j,k}}  \Bigr) \sum_{j > J} \sqrt{j+1} \,2^{-\beta j}
\end{aligned}
\end{equation}
with $C>0$ independent of $J$ and $y$, which implies
\begin{equation}
  \norm{u - u_J}_{L_p(U,C(\bbS^2),\gamma)} \leq C 2^{-( \beta - \delta ) J} 
\end{equation}
with some $C>0$ independent of $J$, for any $\delta \in (0,\beta)$. In other words, with a partial sum of $N = \mathcal{O}(2^{2J})$ terms we obtain an error bound in $L_p(U,C(\bbS^2),\gamma)$ of order $\mathcal{O}(N^{-(\beta-\delta)/2})$ for any $\delta>0$. Note that under the given assumptions, this is precisely the same rate of convergence with respect to the number of terms as guaranteed by \cite[Thm.~2.2]{HLS18} for KL expansions in spherical harmonics \eqref{isotropicreal}.}

{In order to improve the efficiency of the numerical evaluation of \eqref{partsum}, we introduce additional spline approximations of the functions $\psi_{jk}^\mathsf{A}$ as described in Section \ref{sec:splinetrunc} and take advantage of their localisation as in \eqref{eq:truncsupp}. To this end, we replace $u_J$ by
\begin{equation}\label{tildeuJdef}
  \tilde u_J(y) = \sum_{\substack{(j, k) \in  \mathcal{J} \\ j \leq J}} y_{j, k}  \tilde\psi^\mathsf{A}_{jk}, \qquad    \tilde\psi^\mathsf{A}_{jk}(s) = \begin{cases} S_j(d(s, \xi_{jk})) , & d(s,\xi_{jk}) \leq \rho_j, \\ 0, &\text{otherwise,}  \end{cases}
\end{equation}
with suitably chosen $\rho_j \in (0,\pi]$, $j=0,\ldots,J$, and where $S_j \in \mathcal{S}_{n,\tau_j}$ are continuous spline approximations of $R_j\circ \cos$ of piecewise polynomial degree $n\in\N$ with uniform distance $\tau_j>0$ between knots.}

\begin{prop}\label{prop:localizedsum}
{Under the assumptions of Theorem \ref{modifieddecay}, with $\rho_j = 2^{\frac{\beta}{r-2} J} h_j$ and $\tau_j = 2^{-j} 2^{-\frac{\beta}{n+1} ( 1 + \frac{2}{r-2})J}$ for $j \in \N_0$, we have
\[
 \norm{ u - \tilde u_J }_{L_p(U,C(\bbS^2),\gamma)} \leq  C \sum_{j > J} \sqrt{j+1} \,2^{-\beta j}  ,
\]
where $\# \bigl\{ (j,k)\in\mathcal{J} \colon j\leq J, \tilde\psi^\mathsf{A}_{jk}(s) \neq 0\bigr\} \leq C 2^{\frac{2\beta}{r-2} J}$ for each $s \in \bbS^2$.}
\end{prop}

\begin{proof}
{Since $\norm{ u - u_J }_{L_p(U,C(\bbS^2),\gamma)}$ satisfies the stated estimate by \eqref{uJest}, it suffices to obtain the same bound for $\norm{ u_J - \tilde u_J }_{L_p(U,C(\bbS^2),\gamma)}$.
Note that with $\Xi_n$ as in \eqref{Xindef}, 
\[
  \# \bigl\{ (j,k)\in\mathcal{J} \colon j\leq J, \tilde\psi^\mathsf{A}_{jk}(s) \neq 0\bigr\} 
   \leq \sum_{n=1}^{\eta_j + 1} \# \Xi_n \leq c \sum_{n=1}^{\eta_j+1} n \leq  C 2^{\frac{2\beta}{r-2} J}. 
\]
Writing $\rho_j = \eta_j h_j$ with $\eta_j \geq 1$, by the same arguments as in the proof of Corollary \ref{psisumdecay}, we obtain
\[
  \begin{aligned}
 \Bignorm{\sum_{k=1}^{n_j} \bigabs{\psi^\mathsf{A}_{jk} - \tilde\psi^\mathsf{A}_{jk}}}_{C(\bbS^2)}   
   & \leq C \biggl( 2^{\frac{2\beta}{r-2} J} \tau_j^{n+1} 2^{(n+1-\beta) j} + 2^{-\beta j} \sum_{n\geq 0} \frac{\eta_j + n}{1 + (\eta_j + n)^r} \biggr) \\
   & \leq C 2^{-\beta j} \bigl( 2^{-\beta J} + \eta_j^{-(r-2)}  \bigr),
  \end{aligned}
\]
where we have used \eqref{eq:splineapprox} and the choice of $\tau_j$.
With $\omega_{j,k}$ as in \eqref{uJest}, proceeding as there leads to the estimate
\[
  \norm{u_J(y) - \tilde u_J(y)}_{C(\bbS^2)}
  \leq C  \Bigl( \sup_{(j,k) \in \mathcal{J}} \omega_{j,k}^{-1} \abs{y_{j,k}}  \Bigr) \sum_{j = 0}^J \sqrt{j+1} \,2^{-\beta j} \bigl( 2^{-\beta J} + \eta_j^{-(r-2)}  \bigr) .
\]
With $\eta_j = 2^{\frac{\beta}{r-2} J}$, we have
\[
  \sum_{j = 0}^J \sqrt{j+1} \,2^{-\beta j}\bigl( 2^{-\beta J} + \eta_j^{-(r-2)}  \bigr) \leq C \sum_{j= J+1}^{2J+1} \sqrt{j+1} \, 2^{-\beta j},
\]
completing the proof of the first estimate. 
}
\end{proof}

{To conclude, we now discuss the implications of Proposition \ref{prop:localizedsum} on the use of the approximations $\tilde u_J$ for the approximate sampling of the random field $u$.}

\begin{remark}[{Computational costs of sampling}]
{Note that with the cutoff function $\kappa$ chosen as in \eqref{eq:kappaexample}, the maximum admissible value of $r$ in Theorem \ref{modifieddecay} is limited only by the power spectrum $\mathsf{A}$ of any given random field, and $\mathsf{A}$ also determines the value of $\beta$. The total computational costs for using $\tilde u_J$ for approximate random sampling comprise the costs for precomputing certain quantities once, \emph{independently} of any random field; for precomputing the terms in the expansion $\tilde u_J$ for a random field specified by $\mathsf{A}$; and finally, the costs for a single point evaluation of a realization: 
\begin{enumerate}[(i)]
\item As noted in Sec.~5.1, the quadrature weights $\lambda_{jk}$ and points $\xi_{jk}$ need to be precomputed independently of the random field. Suitable high-quality equal-weight point sets (spherical $t$-designs), obtained by specialized optimisation methods, are available in the literature (see \cite{W}). 
\item For each given $\mathsf{A}$, one needs to once precompute the terms in the approximation \eqref{tildeuJdef}, that is, build the spline approximants $S_j$ for $j=0,\ldots, J$. As a consequence of Proposition \ref{prop:localizedsum}, for each $j$ this requires $\mathcal{O}(2^{\beta (\frac{1}{n+1} + \frac{3}{r-2}) J})$ point evaluations of $R_j$, and each such evaluation requires $\mathcal{O}(2^j)$ arithmetic operations.    
\item With this preparation, $\tilde\psi^\mathsf{A}_{jk}(s)$ can be evaluated at unit cost for each $j=0,\ldots,J$, $k=1,\ldots,n_j$, and $s \in \bbS^2$. Since for each $s$, the number of indices $(j,k)$ for which $\tilde\psi^\mathsf{A}_{jk}(s)$ is nonzero is of order $\mathcal{O}(2^{\frac{2\beta}{r-2}J})$, evaluating the approximation of $u(s)$ for arbitrary $s$ requires $\mathcal{O}(2^{\frac{2\beta}{r-2}J})$ operations.
\end{enumerate}
Let us now assume $J$ to be chosen to satisfy $\norm{ u - \tilde u_J }_{L_p(U,C(\bbS^2),\gamma)} \leq \varepsilon$ for a given tolerance $\varepsilon$, then the precomputation of $\tilde u_J$ takes 
\[ \mathcal{O}(\varepsilon^{-\frac{1}{\beta} - \frac{1}{n+1} - \frac{3}{r-2}} \sqrt{1 + \abs{\log \varepsilon}}) \] operations, and evaluating a realization at an arbitrary point on $\bbS^2$  subsequently costs \[ \mathcal{O}(\varepsilon^{-\frac{2}{r-2}}\sqrt{1+\abs{\log\varepsilon}})\] operations.}
\end{remark}

{The truncated needlet-type expansions \eqref{tildeuJdef} are thus suitable for sampling isotropic Gaussian random fields on $\bbS^2$ mainly for large values of $r$. In the case of Mat\'ern-like random fields in \eqref{materndecay} with $A_\ell = c (1 + \ell)^{-2(1+\beta)}$ with $c>0$ for $\ell\geq 0$, the above considerations apply with any $r \in \N$. As a consequence, by using spline approximations of sufficiently high order $n$, one can get arbitrarily close to precomputation costs scaling as $\mathcal{O}(\varepsilon^{-1/\beta})$ and costs per sample that increase more slowly than any negative power of $\varepsilon$ as $\varepsilon\to 0$.
Compared to dedicated methods for sampling of spherical random fields (as considered in \cite{creasey2018fast,emery2019turning,HKS}), this application of the needlet-type expansion may thus be of interest in particular for $\mathsf{A}$ permitting large values of $r$ and when evaluations at highly irregularly spaced or concentrated points on $\bbS^2$ are required.}

\bibliographystyle{amsplain}
\bibliography{Bachmayr_Djurdjevac}

\end{document}